\documentclass[]{amsart}

\usepackage{amsmath}
\usepackage{amssymb}
\usepackage{amscd}
\usepackage{appendix}
\usepackage{amsthm}
\usepackage{amsfonts}
\usepackage[all]{xy}
\usepackage{array}
\usepackage{hyperref}

\theoremstyle{plain}
\newtheorem{thm}{Theorem}[section]
\newtheorem{lem}[thm]{Lemma}
\newtheorem{prop}[thm]{Proposition}
\newtheorem{cor}[thm]{Corollary}

\theoremstyle{definition}
\newtheorem{defn}[thm]{Definition}

\theoremstyle{remark}
\newtheorem{rem}[thm]{Remark}

\renewcommand{\AA}{\mathbb{A}}
\newcommand{\FF}{{\mathbb F}}

\newcommand{\ZZ}{{\mathbb Z}}

\newcommand{\RR}{{\mathbb R}}
\newcommand{\CC}{{\mathbb C}}

\newcommand{\GG}{{\mathbb G}}

\newcommand{\dfk}{\mathfrak{d}}
\newcommand{\mfk}{\mathfrak{m}}
\newcommand{\nfk}{\mathfrak{n}}
\newcommand{\pfk}{\mathfrak{p}}

\newcommand{\yfk}{\mathfrak{y}}

\newcommand{\Afk}{\mathfrak{A}}

\newcommand{\Cfk}{\mathfrak{C}}
\newcommand{\Dfk}{\mathfrak{D}}

\newcommand{\Ifk}{\mathfrak{I}}

\newcommand{\Nfk}{\mathfrak{N}}
\newcommand{\Pfk}{\mathfrak{P}}

\newcommand{\Ecal}{\mathcal{E}}
\newcommand{\Fcal}{\mathcal{F}}

\newcommand{\Mcal}{\mathcal{M}}

\newcommand{\Ocal}{\mathcal{O}}

\newcommand{\Ccal}{\mathcal{C}}
\newcommand{\Dcal}{\mathcal{D}}

\newcommand{\Kcal}{\mathcal{K}}
\newcommand{\Xcal}{\mathcal{X}}
\newcommand{\Zcal}{\mathcal{Z}}

\newcommand{\Mat}{\operatorname{Mat}}

\newcommand{\ord}{\operatorname{ord}}

\newcommand{\GL}{\operatorname{GL}}
\newcommand{\SL}{\operatorname{SL}}

\newcommand{\Spp}{\operatorname{Sp}}
\newcommand{\Oo}{\operatorname{O}}

\newcommand{\Hom}{\operatorname{Hom}}

\newcommand{\Pic}{\operatorname{Pic}}

\newcommand{\End}{\operatorname{End}}
\newcommand{\Sym}{\operatorname{Sym}}
\newcommand{\re}{\operatorname{Re}}
\newcommand{\Sp}{\operatorname{sp}}

\newcommand{\Hasse}{\operatorname{Hasse}}
\newcommand{\Diff}{\operatorname{Diff}}
\newcommand{\Spec}{\operatorname{Spec}}
\newcommand{\Mor}{\operatorname{Mor}}
\newcommand{\Sch}{\operatorname{\it Sch}}
\newcommand{\SP}{\operatorname{SP}}
\newcommand{\Char}{\operatorname{char}}
\newcommand{\pr}{\operatorname{pr}}
\newcommand{\Ob}{\operatorname{Ob}}

\numberwithin{equation}{section}

\begin{document}

\title{On derivatives of Siegel-Eisenstein series over global function fields}
\author[Fu-Tsun Wei]{Fu-Tsun Wei}
\address{Institute of Mathematics, Academia Sinica, Taiwan}
\email{ftwei@math.sinica.edu.tw}


\begin{abstract}
The aim of this article is to study the derivative of \lq\lq incoherent\rq\rq\ Siegel-Eisenstein series on symplectic groups over function fields. By the Siegel-Weil formula for \lq\lq coherent\rq\rq\ Siegel-Eisenstein series, we can relate the non-singular Fourier coefficients of the derivative in question to the arithmetic of quadratic forms. Restricting to the special case when the incoherent quadratic space has dimension $2$, we explicitly compute all the Fourier coefficients, and connect the derivative with the special cycles on the coarse moduli schemes of rank $2$ Drinfeld modules with \lq\lq complex multiplication.\rq\rq 
\end{abstract}

\maketitle

\section*{Introduction}\label{Intro}

The Siegel-Eisenstein series on symplectic groups are closely related to the arithmetic of quadratic forms.
By the Siegel-Weil formula (cf.\ \cite{We2}, \cite{KR2}, and \cite{Wei2}), certain special values of \lq\lq coherent\rq\rq\ Siegel-Eisenstein series are expressed in terms of the theta series associated with coherent quadratic spaces. This can be applied to the special values of automorphic $L$-functions via integral representations, e.g.\ Rankin-Selberg method or the Garrett-type representation for triple product $L$-functions (cf.\ \cite{Zha}, \cite{G-K}, and \cite{Wei1}). Using the Siegel-Weil formula, the central critical derivatives of \lq\lq incoherent\rq\rq\ Siegel-Eisenstein series can be also understood by theta series. Furthermore, in the number field case, the work of Kudla, Rapoport, and Yang (cf.\ \cite{Kud}, \cite{Kud2}, and \cite{KRY}) provides evidences of the relations between the derivatives in question and arithmetic cycles on Shimura varieties of orthogonal type. More precisely, the non-singular Fourier coefficients of such derivatives are essentially the \lq\lq degree\rq\rq\ of the corresponding cycles. This observation can be viewed as an analogue of the Siegel-Weil formula for the derivatives, which is a key ingredient in the work of Yuan-Zhang-Zhang \cite{YZZ} on the central critical derivatives of triple product $L$-functions. The aim of this article is to study the derivatives of incoherent Siegel-Eisenstein series over function fields, and connect the non-singular Fourier coefficients with arithmetic cycles on the moduli space of Drinfeld modules in a simple case. The result obtained in this paper provides an evidence in the function field setting of the phenomenon first observed by Kudla.\\

Let $k$ be a global function field with constant field $\FF_q$. Suppose $q$ is odd. Take a positive odd integer $n$ and let $\Ccal = \{\Ccal_v\}_v$ be an \text{\it incoherent} quadratic space over $k$ (defined in Section~\ref{Preli.Coh}) with dimension $n+1$. For each Schwartz function $\varphi$ on $\Ccal^n$,
the associated Siegel-Eisenstein series $E(\cdot, s,\Phi_{\varphi})$ on the symplectic group $\Spp_n(\AA_k)$ (where $\AA_k$ is the adele ring of $k$) always vanishes at the central critical point $s=0$ (cf.\ Theorem~\ref{thm: Sie-Eis.2}). To explore the central critical derivative, we first restrict ourselves to the special case when $n = 1$.
Then the associated quadratic character $\chi_{\Ccal}$ of $\Ccal$ on $k^{\times} \backslash \AA_k^{\times}$ must be non-trivial. Let $K$ be the quadratic extension of $k$ corresponding to the kernel of $\chi_{\Ccal}$ via class field theory. We take a place $\infty$ of $k$ not split in $K$ and $\alpha \in k^{\times}$ such that $\Ccal_v$ is isomorphic to $(K_v, \alpha \cdot N_{K/k})$ for every place $v \neq \infty$. Here $K_v = K \otimes_k k_v$ and $k_v$ is the completion of $k$ at $v$; $N_{K/k}$ is the norm form of $K$ over $k$. 
Choose a particular Schwartz function $\varphi^{(\alpha)} \in S(\Ccal(\AA_k))$ so that the associated \lq\lq Siegel section\rq\rq\ $\Phi_{\varphi^{(\alpha)}}$ is a \lq\lq new\rq\rq\ vector in the space $R_1(\Ccal)$ consisting all of the sections $\Phi_{\varphi}$ for $\varphi \in S(\Ccal(\AA_k))$ (see (\ref{eqn: Examp.1}) for the precise definition of $\varphi^{(\alpha)}$).
We are interested in the central critical derivative
$$\eta^{(\alpha)}:= \frac{\partial}{\partial s} \widetilde{E}(\cdot,s,\Phi_{\varphi^{(\alpha)}}) \Big|_{s=0},$$
where $\widetilde{E}(\cdot,s,\Phi_{\varphi^{(\alpha)}})$ is the Siegel-Eisenstein series associated with $\varphi^{(\alpha)}$ modified by the Hecke $L$-function of $\chi_{\Ccal}$ (see (\ref{eqn: Examp.2})).
Via Rankin-Selberg method, $\widetilde{E}(\cdot,s,\Phi_{\varphi^{(\alpha)}})$ shows up in the study of Rankin-type $L$-functions associated with \lq\lq Drinfeld type\rq\rq\ automorphic forms on $\GL_2(\AA_k)$ (cf.\ \cite{CWY}). This is our motivation to first target at this particular function $\eta^{(\alpha)}$.
For each $\beta \in k^{\times}$, we show that the $\beta$-th Fourier coefficients of $\eta^{(\alpha)}$ are expressed by representation numbers of $\beta$ as the corresponding norm forms (cf.\ Proposition~\ref{prop: Examp.9}).
In particular, these non-zero Fourier coefficients of $\eta^{(\alpha)}$ are related to special cycles on the \lq\lq compactification\rq\rq\ $\Xcal_{O_K}$ of the coarse moduli scheme of rank one Drinfeld $O_K$-modules (where $O_K$ is the ring of functions in $K$ regular outside $\infty$). To be more concrete, the main result of this paper is in the following.

\begin{thm}\label{thm: Intro.1}
\text{\rm (cf.\ Theorem~\ref{thm: Alg-geo.13})}
For each $y \in \AA_k^{\times}$ and $\beta \neq 0$, there exists an algebraic $0$-cycle $\mathbf{z}(y,\beta)$ on $\Xcal_{O_K}$ such that the  Fourier coefficient $\eta_\beta^{(\alpha),*}(y)$ is equal to:
$$- \frac{\chi_{\Ccal}(y)|y|_{\AA_k}}{f_{\infty} \cdot \#\Pic(A)} \cdot \deg \mathbf{z}(y,\beta),$$
where $f_{\infty}$ is the residue degree of $\infty$ in $K/k$, and $A$ is the ring of functions in $k$ regular outside $\infty$.
\end{thm}

The above theorem can be viewed as a function field analogue of Kudla-Rapoport-Yang's result in \cite{KRY}.
The moduli problem of Drinfeld modules are recalled  in Section~\ref{sec: Alg-geo.1}, and we refer the readers to \cite{Dri} and \cite{Lau} for further details.
The algebraic cycle $\mathbf{z}(y,\beta)$ is constructed via \lq\lq special morphisms\rq\rq\ defined in the following. Let $(L,\phi)$ and $(L',\phi')$ be rank one Drinfeld $O_K$-modules over a scheme $S$. Then $\phi$ (resp.\ $\phi'$) induces a left action of $\Mat_2(O_K)$ on $L^{\oplus 2}$ (resp.\ $L^{\prime \oplus 2}$). Let $\Dcal_{\alpha}$ be the quaternion algebra $K + Kj_{\alpha}$ over $k$ where $j_{\alpha}^2 = -\alpha$ and $j_{\alpha} a = \bar{a} j_{\alpha}$ (here $(a\mapsto \bar{a})$ is the non-trivial automorphism on $K$ over $k$).
The reduced norm form on $\Dcal_{\alpha}$ can be expressed by $N_{K/k}\oplus (\alpha \cdot N_{K/k})$. 
We fix a $K$-algebra isomorphism $K\otimes_k \Dcal_{\alpha} \cong \Mat_2(K)$ defined by:
$$ 
\begin{tabular}{rcll}
$a \otimes 1$ & $\longmapsto$ & $\begin{pmatrix} a & 0 \\ 0 & a \end{pmatrix}$, & $\forall a \in K$, \\
$ 1 \otimes (a_1 + a_2 j_{\alpha})$ & $\longmapsto$ & $\begin{pmatrix} a_1 & a_2 \\ -\alpha \overline{a_2} & \overline{a_1} \end{pmatrix}$, & $\forall a_1 + a_2 j_{\alpha} \in \Dcal_{\alpha}$.
\end{tabular}$$
A \text{\it special morphism} $f : (L^{\oplus 2},\phi) \rightarrow (L^{\prime \oplus 2},\phi')$ is a morphism from $L^{\oplus 2}$ to $L^{\prime \oplus 2}$ (as group schemes over $S$) satisfying
$$ f \phi_{1\otimes d} = \phi'_{1\otimes d} f \quad \text{and} \quad f \phi_{a\otimes 1} = \phi'_{\overline{a}\otimes 1} f, \quad \forall a \in O_K \text{ and } d \in O_{\Dcal_{\alpha}} := \Dcal_{\alpha} \cap \Mat_2(O_K).$$

When $\Dcal_{\alpha}$ splits at $\infty$ (i.e.\ $\Dcal_{\alpha} \otimes_k k_{\infty} \cong \Mat_2(k_{\infty})$), one can associate $(L^{\oplus 2},\phi)$ (resp.\ $(L^{\prime \oplus 2},\phi')$) a $\Dcal_{\alpha}$-elliptic sheave $\Ecal_{\phi}$ (resp.\ $\Ecal_{\phi'}$) with \lq\lq complex multiplication\rq\rq\ by $O_K$ via $\phi$ (resp.\ $\phi'$, cf.\ \cite{LRS} and \cite{Pap}).
Then the set of special morphisms can be identified with
$$\{f \in \Hom(\Ecal_{\phi},\Ecal_{\phi'}) : f \phi_a = \phi_{\overline{a}} f, \quad \forall a \in O_K\}.$$
Here the homomorphisms between $\Dcal_{\alpha}$-elliptic sheaves is allowed to \lq\lq shift the indices\rq\rq\ (cf.\ \cite{Pap} Definition 4.2).
Moreover, when $S = \Spec(\kappa)$ where $\kappa$ is a perfect field,
$(L^{\oplus 2},\phi)$ and $L^{\prime \oplus 2},\phi')$ are Anderson's \lq\lq pure abelian $A$-modules\rq\rq\ (cf.\ \cite{And}) with $O_K \otimes_A O_{\Dcal_{\alpha}}$ multiplication, which can be viewed as analogue of abelian surfaces with quaternion and complex multiplication.
In particular when $\alpha = -1$, we have $\Dcal_{\alpha} \cong \Mat_2(k)$ and the set of special morphisms can be realized by
$$\{f \in \Hom_A\big((L,\phi),(L',\phi')\big) \mid f \phi_{a\otimes 1} = \phi'_{\overline{a}\otimes 1} f, \quad \forall a \in O_K\}.$$

We remark that the constant Fourier coefficient $\eta_0^{(\alpha),*}(y)$ are related to the \lq\lq Taguchi height\rq\rq\ of rank $2$ Drinfeld $A$-modules with \lq\lq complex mulitplication\rq\rq\ by $O_K$. The Taguchi height of a Drinfeld module (cf.\ \cite{Tag} and \cite{Wei3}) is viewed as a natural analogue of the Faltings height of abelian varieties over number fields. Let $\phi$ be a Drinfeld module over $\overline{k}$ of rank $2$ with complex mulitplication by $O_K$. Comparing the formula of $\eta_0^{(\alpha),*}(y)$ in Lemma~\ref{lem: Examp.2} and the Taguchi height of $\phi$ in \cite[Corollary 0.2]{Wei3}, we then express $\eta_0^{(\alpha),*}(y)$ as (cf.\ Lemma~\ref{lem: Alg-geo.14}):
\begin{eqnarray}
\eta^{(\alpha),*}_{0}(y) &=& \frac{2\chi_K(y)|y|_{\AA_k}\#\Pic(O_K)}{f_\infty \#\Pic(A)} \nonumber \\
&& \!\!\!\!\!\!\!\!\!\!\!\!\!\!\! \cdot \left[ \ln|y|_{\AA_k} - 2\tilde{h}_{\text{Tag}}(\phi) -(g_k-1)\ln q  - \frac{\zeta_A'(0)}{\zeta_A(0)} 
+  \frac{(-1)^{f_\infty} q^{\deg \infty} +1 - 2^{f_\infty}}{f_\infty(q^{\deg \infty}+1)} \cdot \deg \infty \ln q\right]. \nonumber
\end{eqnarray}
Here $\zeta_A(s):= \prod_{v \neq \infty} (1-q^{-s \deg v})^{-1}$ is the zeta function of $A$.
This provides a geometric intepretation of the constant Fourier coefficient $\eta_0^{(\alpha),*}(y)$.\\



The description of the non-zero Fourier coefficients of $\eta^{(\alpha)}$ in Proposition~\ref{prop: Examp.9} is in fact derived from a general pattern in Theorem~\ref{thm: Deriv.1}. We give a short discussion in the following. Let $n$ be an arbitrary positive odd integer and $\Ccal$  be an incoherent quadratic space with dimension $n+1$.
for each $\beta \in \Sym_n(k)$ (i.e.\ $\beta$ is symmetric) with $\det \beta \neq 0$, let
$$\Diff(\beta,\Ccal):= \big\{ \text{place $v$ of $k$} \ \big|\  \Hasse_v(\Ccal) \neq \chi_{\Ccal_v}(\det \beta)\cdot (\det \beta, -(-1)^{(n+1)/2})_v\cdot \Hasse_v(\beta)\big\},$$
where $\Hasse_v(\Ccal)$ (resp.\ $\Hasse_v(\beta)$) is the Hasse invariant of $\Ccal$ (resp.\ $\beta$) at the place $v$ of $k$, and $(\cdot,\cdot)_v$ is the local Hilbert quadratic symbol at $v$. It can be shown that the cardinality of $\Diff(\beta,\Ccal)$ provides a lower bound for the vanishing order of $\beta$-th Fourier coefficient of $E(\cdot,s,\Phi_{\varphi})$ at $s=0$ (cf.\ Proposition~\ref{prop: Fourier.6}).
When $\Diff(\beta,\Ccal) = \{v_0\}$, let $V_{\beta}$ be the coherent quadratic space over $k$ whose associated character $\chi_{V_{\beta}} = \chi_{\Ccal}$ and the Hasse invariant $\Hasse_v(V_{\beta}) = \Hasse_v(\Ccal)$ for every place $v \neq v_0$.
Then the $\beta$-th Fourier coefficient of the central critical derivative can be understood by the theta series associated with $V_{\beta}$: 

\begin{thm}\label{thm: Intro.2}
\text{\rm (cf.\ Theorem~\ref{thm: Deriv.1})}
Let $\Ccal$ be an incoherent quadratic space over $k$ with even dimension $m$ and take $n = m-1$. For $\beta \in \Sym_n(k)$ with $\det \beta \neq 0$, suppose $\Diff(\beta,\Ccal) = \{v_0\}$ and the associated quadratic space $V_{\beta}$ is anisotropic.
Then for each pure-tensor Schwartz function $\varphi = \otimes_v\varphi_v \in S(\Ccal(\AA_k)^n)$ and $a \in \GL_n(\AA_k)$, the Fourier coefficient
$$\frac{\partial}{\partial s} E_{\beta}^*(a,s,\Phi_{\varphi}) \bigg|_{s=0} = 2\cdot \frac{W_{v_0,a_{v_0}*\beta}'(0,\Phi_{v_0,\varphi_{v_0}})}{W_{v_0,a_{v_0}*\beta}(0,\Phi_{v_0,\widetilde{\varphi}_{v_0}})}\cdot \Theta_{\beta}^*(a,\widetilde{\varphi}),$$
where $a_{v_0}*\beta = {}^ta_{v_0}\beta a_{v_0}$, $\widetilde{\varphi} = \otimes_v \widetilde{\varphi}_v \in S(V_{\beta}(\AA_k)^n)$ is any pure-tensor Schwartz function so that:
\begin{itemize}
\item[(i)] for $v \neq v_0$, $\widetilde{\varphi}_v = \varphi_v$ $($here we identify $V_{\beta,v}$ with $\Ccal_v)$;
\item[(ii)] the Whittaker function $W_{v_0,a_{v_0}*\beta}(s,\Phi_{v_0,\widetilde{\varphi}_{v_0}})$ associated with $\widetilde{\varphi}_{v_0} \in S(V_{\beta,v_0}^n)$ does not vanish at $s = 0$;
\end{itemize}
and
$\Theta_{\beta}^*(a,\widetilde{\varphi})$ are Fourier coefficients of the theta series $\Theta(g,\widetilde{\varphi})$ introduced in \text{\rm Theorem~\ref{thm: Sie-Eis.1}}.
\end{thm}

We refer the readers to Section \ref{Deriv} for further details.
The flexibility of the choice of $\widetilde{\varphi}_{v_0}$ is beneficial to the calculation of the values $W_{v_0,a_{v_0}*\beta}(0,\Phi_{v_0,\widetilde{\varphi}_{v_0}})$ and $\Theta_{\beta}^*(a,\widetilde{\varphi})$. Moreover, one has
$$\Theta_{\beta}^*(a ,\widetilde{\varphi}) = 
\chi_{\Ccal}(\det a)|\det a|_{\AA_k}^{\frac{n+1}{2}} \cdot \int_{\Oo(V_{\beta})(k)\backslash \Oo(V_{\beta})(\AA_k)} \sum_{x \in V_{\beta}^n, \atop Q_{V_{\beta}}(x) = \beta} \widetilde{\varphi}(h^{-1}xa) dh,$$
where $\Oo(V_{\beta})$ is the orthogonal group of $V_{\beta}$, $Q_{V_{\beta}}$ denotes the quadratic form on $V_{\beta}$, and $dh$ is the $\Oo(V_{\beta})(\AA_k)$-invariant measure having total mass $1$.
This observation says that the $\beta$-th Fourier coefficients of the central critical derivative is essentially the representation number of $\beta$ as the quadratic form $Q_{V_{\beta}}$.
The assumption of $V_{\beta}$ being anisotropic in Theorem~\ref{thm: Intro.2} is due to the restriction of the function field analogue of Siegel-Weil formula in \cite{Wei2}. We can remove this assumption if a stronger version of Siegel-Weil formula over function fields is verified, which will be explored in a future work.\\

The structure of this article is organized as follows. We fix our basic notations in Section~\ref{Preli.bas} and review the concept of the incoherent quadratic spaces in Section~\ref{Preli.Coh}. In Section~\ref{Wrepn}, we recall the Weil representation on the spaces of Schwartz functions on quadratic spaces over $k$ and the relevant results we need. In Section~\ref{Sie-Eis}, we introduce the Siegel-Eisenstein series on symplectic groups, and recall their analytic properties. Also included in Section~\ref{Sie-Eis} is the Siegel-Weil formula for the anisotropic case over function fields. In Section~\ref{Fourier}, we focus on the non-singular Fourier coefficients of incoherent Siegel-Eisenstein series, and determine their vanishing order at the central critical point by Whittaker functions. In Section~\ref{Deriv}, we study the derivative of the non-singular Fourier coefficients of incoherent Siegel-Eisenstein series, and prove Theorem~\ref{thm: Intro.2} via the Siegel-Weil formula. In Section~\ref{Examp}, we concentrate on the special case when the incoherent quadratic space $\Ccal$ has dimension $2$, and compute explicitly all the Fourier coefficients of $\eta^{(\alpha)}$. Finally, the algebraio-geometric aspects of $\eta^{(\alpha)}$ is discussed in Section~\ref{Alg-geo}. We recall the moduli problem of Drinfeld modules in Section \ref{sec: Alg-geo.1}, and introduce the special morphisms between Drinfeld modules in Section~\ref{sec: Alg-geo.2}. The proof of Theorem~\ref{thm: Intro.1} and the formula of the constant Fourier coefficient $\eta_0^{(\alpha),*}(y)$ are given in Section~\ref{sec: Alg-geo.3}.\\


\section{Preliminary}\label{Preli}

\subsection{Basic setting}\label{Preli.bas}

Let $\FF_q$ be the finite field with $q = p^{r_0}$ elements.
Let $k$ be a global function field with constant field
$\mathbb{F}_q$, i.e.\ $k$ is a finitely generated field extension over $\mathbb{F}_q$ with transcendence degree one and $\mathbb{F}_q$ is algebraically closed in $k$. 
In this article we always assume that $q$ is \text{\bf odd}.
For each place $v$ of $k$, let $k_v$ be the completion of $k$ at $v$, and $O_v$ denotes the valuation ring in $k_v$. 
Choosing a uniformizer $\pi_v$ in $O_v$, we set $\FF_v := O_v/\pi_vO_v$, the residue field at $v$, and $q_v$ to be the cardinality of $\FF_v$. Put $\deg v := [\FF_v: \FF_q]$, called the degree of $v$.
The absolute value on $k_v$ is normalized to be:
$$|a_v|_v:= q_v^{-\ord_v(a_v)} = q^{-\deg v\ord_v(a_v)}, \quad \forall a_v \in k_v.$$
Let $\AA_k$ be the adele ring of $k$ and set $O_{\AA_k}:= \prod_v O_v$, the maximal compact subring of $\AA_k$.
For each element $a = (a_v)_v$ in the idele group $\AA_k^{\times}$, the norm $|a|_{\AA_k}$ is defined as:
$$|a|_{\AA_k}:= \prod_v |a_v|_v.$$
Embedding $k$ (resp.\ $k^{\times}$) into $\AA_k$ (resp.\ $\AA_k^{\times}$) diagonally, we have the product formula: $|\alpha|_{\AA_k} = 1$ for every $\alpha \in k^{\times}$.
Throughout this article, fix a non-trivial additive character $\psi :\AA_k \rightarrow \CC^{\times}$ which is trivial on $k$.
For each place $v$ of $k$, let $\psi_v$ be the additive character on $k_v$ defined by $\psi_v(a_v):= \psi(0,...,0,a_v,0,...)$ for each $a_v$ in $k_v$. We denote $\delta_v$ to be the \lq\lq conductor of $\psi$ at $v$,\rq\rq\ i.e.\ the maximal integer $r$ such that $\pi_v^{-r}O_v$ is contained in the kernel of $\psi_v$.
It is known that $\sum_v \delta_v \deg v = 2g_k-2$,
where $g_k$ is the genus of $k$.\\

Take a place $v$ of $k$.
Let $(V, Q_V)$ be a non-degenerate quadratic space of dimension $m$ over $k_v$. The bilinear form on $V\times V$ is
$$\langle x,y \rangle_V:= Q_V(x+y)-Q_V(x)-Q_V(y), \quad \forall x,y \in V.$$ 
The associated quadratic character $\chi_{V}$ on $k_v^{\times}$ is defined by
$$\chi_{V}(\alpha_v) := (\alpha_v, (-1)^{\frac{m(m+1)}{2}} \det V)_v, \quad \forall \alpha_v \in k_v^{\times}.$$
Here 
$(\cdot,\cdot)_v: k_v^{\times}/(k_v^{\times})^2 \times k_v^{\times}/(k_v^{\times})^2 \rightarrow \{\pm 1\}$ is the Hilbert quadratic symbol, i.e.\
$$(a_v,b_v)_v = \begin{cases}
1, & \text{ if $a_v X^2 + b_v Y^2 = Z^2$ has non-trivial solutions,}\\
-1, & \text{ otherwise;}
\end{cases}$$
and $\det V \in k_v^{\times}/(k_v^{\times})^2$ is the discriminant of $V$, i.e. 
$$\det V := \det (\langle x_i,x_j \rangle_V)_{1\leq i, j \leq m} \quad \text{for every basis $\{x_1,...,x_m\}$ of $V$}.$$ 

Take a basis $\{x_1,...,x_m\}$ of $V$ such that the quadratic form becomes 
$$Q_V(x) = \sum_{i=1}^m c_i a_i^2, \quad \forall x = \sum_{i=1}^m a_i x_i \in V,$$
where $c_i \in k_v^{\times}$ for $i = 1,...,m$.
The Hasse invariant of $V$ is:
$$\Hasse_v(V):= \prod_{1\leq i < j \leq m} (c_i,c_j)_v.$$

For a global non-degenerate quadratic space $W$ over $k$, set $\chi_W:= \otimes \chi_{W_v}: k^{\times} \backslash \AA_k^{\times} \rightarrow \{\pm 1\}$, where $W_v = W\otimes_k k_v$, and call $\Hasse_v(W_v)$ the Hasse invariant of $W$ at $v$.

\subsection{Incoherent quadratic spaces}\label{Preli.Coh}

Fix a character $\chi = \otimes_v \chi_v: k^{\times} \backslash \AA_k^{\times} \rightarrow \{\pm 1\}$. An  \it incoherent quadratic space $\Ccal = \{\Ccal_v\}_v$ over $k$ of dimension $m$ with character $\chi$ \rm is a collection of non-degenerate quadratic spaces $\Ccal_v$ over $k_v$, indexed by the places of $k$, such that
\begin{itemize}
\item[(i)] $\dim_{k_v}\Ccal_v = m$ and $\chi_{\Ccal_v} = \chi_v$ for every place $v$. 
\item[(ii)] There exists a global non-degenerate quadratic space $W$ of dimension $m$ over $k$ with character $\chi_W = \chi$ such that $W_v \cong \Ccal_v$ for almost all $v$.
\item[(iii)] (Incoherence condition) The product formula fails for the Hasse invariants:
$$\prod_v \Hasse_v(\Ccal_v) = -1.$$
\end{itemize}

\begin{rem}
Let $\Ccal$ be an incoherent quadratic space over $k$. Due to the incoherence condition, there is no quadratic spaces $W$ over $k$ satisfying that $W_v \cong \Ccal_v$ for all $v$. Conversely,
suppose a \lq\lq coherent\rq\rq\ collection $\Ccal$ of quadratic spaces is given, i.e. $\Ccal$ satisfies (i), (ii), and $\prod_v \Hasse_v(\Ccal_v) = 1$. Then we can find a unique (up to isomorphism) quadratic space $W$ over $k$ such that $W_v \cong \Ccal_v$ for every $v$. Thus a non-degenerate quadratic space over $k$ is also called a \text{\it coherent} quadratic space.
\end{rem}

\section{Weil representation on Schwartz spaces}\label{Wrepn}

Given a positive integer $n$, let $\Spp_n$ be the symplectic group of degree $n$, i.e.\
$$\Spp_n := \left\{ g \in \GL_{2n} \ \Bigg| \ {}^tg \begin{pmatrix}0 & I_n \\ -I_n & 0\end{pmatrix} g = \begin{pmatrix}0 & I_n \\ -I_n & 0\end{pmatrix}\right\}.$$
We view $\Spp_n$ as an affine algebraic group over $\FF_q$.
The \it Siegel-parabolic subgroup \rm $P_n$ is the parabolic subgroup $M_n\cdot N_n$,
where
$$M_n:= \left\{ \mfk(a) = \begin{pmatrix} a & 0 \\ 0 & {}^ta^{-1} \end{pmatrix} \text{ } \bigg| \text{ } a \in \GL_n \right\} \text{ \ and \ }
N_n:= \left\{ \nfk(b) = \begin{pmatrix} I_n & b \\ 0 & I_n \end{pmatrix} \text{ } \bigg| \text{ } b = {}^tb \in \Mat_n\right\}.$$ 
Set $\Sym_n :=\{ b = {}^tb \in \Mat_n\}$. By the map $\mfk$ and $\nfk$, $\GL_n$ and $\Sym_n$ are isomorphic to $M_n$ and $N_n$ respectively.

Take a place $v$ of $k$.
Let $(V, Q_V)$ be a non-degenerate quadratic space over $k_v$ with even dimension $m$. 
The orthogonal group of $V$ is denoted by $\Oo(V)$, i.e.\ 
$$\Oo(V) = \{ h \in \GL(V) \mid Q_V(hx) = Q_V(x), \quad \forall x \in V\}.$$
Let $S(V^n)$ be the space of \it Schwartz functions on $V^n$\rm, i.e.\ the space of $\CC$-valued functions on $V^n$ which are locally constant and compactly supported.
The \it $($local$)$ Weil representation \rm $\omega_v (= \omega_{v, \psi_v})$ of $\Spp_n(k_v) \times \Oo(V)$ on $S(V^n)$
is determined by the following:
for every $\varphi_v \in S(V^n)$ and $x \in V^n$,
\begin{eqnarray}
 (\omega_v(h) \varphi_v)(x) &:=& \varphi_v(h^{-1}x_1,...,h^{-1}x_n), \text{ }\forall h \in \Oo(V); \nonumber \\
 \left(\omega_v\begin{pmatrix} a&0 \\ 0 &{}^t a^{-1}\end{pmatrix}\varphi_v\right)(x)
&:=& \chi_{V}(\det a)|\det a|_{v}^{m/2} \cdot \varphi_v(x\cdot a), \text{ }\forall a \in \GL_n(k_v); \nonumber \\
\left(\omega_v\begin{pmatrix} I_n & b\\0&I_n\end{pmatrix} \varphi_v\right)(x)
&:=& \psi_v \Big(\text{Trace}\big(b \cdot Q_V^{(n)}(x)\big)\Big) \cdot \varphi_v(x), \text{ }\forall b \in \Sym_n(k_v); \nonumber\\
\left(\omega_v\begin{pmatrix}0 & I_n \\ -I_n & 0\end{pmatrix}\varphi_v\right)(x) &:=& \varepsilon_v(V)^n \cdot \widehat{\varphi}_v(x). \nonumber
\end{eqnarray}
Here:
\begin{itemize}
\item $Q_V^{(n)}:V^n \rightarrow \Sym_n(k_v)$ is the {\it moment map,} i.e.\ for any $x = (x_1,...,x_n) \in V^n$,
$$Q_V^{(n)}(x) = \left(\frac{1}{2} \langle x_i,x_j \rangle_V\right)_{1\leq i,j\leq n}.$$
\item $\widehat{\varphi}_v$ is the Fourier transform of $\varphi_v$: 
$$\widehat{\varphi}_v(x):= \int_{V^n} \varphi_v(y) \cdot \psi_v(\sum_{i=1}^n  \langle x_i,y_i \rangle_V) dy, \quad \forall x = (x_1,...,x_n) \in V^n.$$
The Haar measure $dy = dy_1\cdots dy_n$ is chosen to be \it self dual\rm, i.e.\ 
$$\widehat{\widehat{\varphi}}_v (x) = \varphi_v(-x), \quad \forall x \in V^n.$$
\item $\varepsilon_v(V)$ is the \text{\it Weil index of $V$}, i.e.\
$$\varepsilon_v(V):= \int_{L} \psi_v\big(Q_V(x)\big) dx$$
for any sufficiently large $O_v$-lattice $L$ in $V$. The Haar measure $dx$ is also chosen to be self dual with respect to the pairing $\psi_v(\langle \cdot,\cdot \rangle_V)$.
\end{itemize}
${}$

Given a character $\chi_v: k_v^{\times} \rightarrow \CC^{\times}$,
let $I_v(s,\chi_v)$ be the space of locally constant $\CC$-valued functions $f_v$ on $\Spp_n(k_v)$ satisfying that
$$f_v(\nfk(b) \mfk(a) g) = \chi_v(\det a)|a|_v^{s+\frac{n+1}{2}} f_v(g), \quad \forall a \in \GL_n(k_v),\ b \in \Sym_n(k_v),\ g \in \Spp_n(k_v).$$
The action of $\Spp_n(k_v)$ on $I(s,\chi_v)$ is defined by right translation.
Let $\Phi_v$ be the following $\Spp_n(k_v)$-equivariant homomorphism from $S(V^n)$ to $I_v((m-n-1)/2,\chi_V)$:
$$ \varphi_v \longmapsto \Phi_{v,\varphi_v}:= \Big(g \mapsto \big(\omega_v(g) \varphi_v\big)(0)\Big).$$
Denote by $R_n(V)$ the image of $\Phi_v$. We then have:

\begin{thm}\label{thm: Wrepn.1}
$(1)$ Let $V$ be a non-degenerate quadratic space over $k_v$ with even dimension. Then $R_n(V)$ is isomorphic to $S(V^n_0)$, where $V^n_0 = \{x \in V^n: Q_V^{(n)}(x) = 0\}$. Moreover, the isomorphism is induced from the restriction of $\Phi_v$ on $S(V^n_0)$. \\
$(2)$ Let $\chi_v: k_v^{\times} \rightarrow \{\pm 1\}$ be a quadratic character and $n$ is a positive odd integer. Then
$$I_v(0,\chi_v) = R_n(V^+) \oplus R_n(V^-)$$
where $V^{\pm}$ is the non-degenerate quadratic space over $k_v$ of dimension $n+1$ such that $\chi_v = \chi_{V^{\pm}}$ and $\Hasse_v(V^{\pm}) = \pm 1$.
\end{thm}

\begin{proof}
$(1)$ follows from \cite[Proposition B.1]{Wei2} and $(2)$ follows from \cite[Corollary 3.7]{KR1}. Although the characteristic of the base local field is assumed to be $0$ in \cite{KR1}, the argument actually works for the case of odd characteristic. We recall the steps of the proof for $(2)$ in the following.

By the Frobenius reciprocity theorem, we know $\dim_{\CC} \End_{\Spp_n(k_v)}\big(I(0,\chi_v)\big) \leq 2$. In particular, $\dim_{\CC} \End_{\Sp_n(k_v)}\big(I(0,\chi_v)\big) = 1$ if $\chi_v \equiv 1$ and $\dim V = 2$.
Using the twisted Jacquet functor on $R_n(V^+)$ and $R_n(V^-)$, we obtain that $R_n(V^+)$ and $R_n(V^-)$ are inequivalent submodules of $I(0,\chi_v)$ and
$$R_n(V^+) \not\subset R_n(V^-), \quad R_n(V^-) \not\subset R_n(V^+).$$
Moreover, using the inner product 
$$(f_1,f_2) := \int_{P_n(k_v)\backslash \Spp_n(k_v)} f_1(g)\overline{f_2(g)} dg$$ on $I(0,\chi_v)$, we obtain that $I(0,\chi_v)$ is completely reducible. Therefore the result holds.
\end{proof}

Now, let $W$ (resp.\ $\Ccal$) be a coherent (resp.\ incoherent) quadratic space over $k$ with even dimension $m$. We have the \text{\it global Weil representation} $\omega := \otimes_v \omega_v$ on the Schwartz space $S(W(\AA_k)^n)$ (resp.\ $S(\Ccal(\AA_k)^n)$). Here $W(\AA_k):= W \otimes_k \AA_k = \prod_v^{\prime} W_v$ and $\Ccal(\AA_k):= \prod^{\prime}_v \Ccal_v$. Let $\chi = \chi_W$ (resp.\ $\chi_{\Ccal}$) and $I(s,\chi)$ denotes the space of locally constant $\CC$-valued functions $f$ on $\Spp_n(\AA_k)$ satisfying that
$$f(\nfk(b) \mfk(a) g) = \chi(\det a)|a|_{\AA_k}^{s+\frac{n+1}{2}} f(g), \quad \forall a \in \GL_n(\AA_k),\ b \in \Sym_n(\AA_k),\ g \in \Spp_n(\AA_k).$$
Let $\Phi$ be the $\Spp_n(\AA_k)$-equivariant homomorphism from $S(W(\AA_k)^n)$ (resp.\ $S(\Ccal(\AA_k)^n)$ to $I_v((m-n-1)/2,\chi)$ defined by:
$$ \varphi \longmapsto \Phi_{\varphi}:= \Big(g \mapsto \big(\omega(g) \varphi\big)(0)\Big).$$
Denote by $R_n(W)$ (resp.\ $R_n(\Ccal)$) the image of $\Phi$. Then Theorem \ref{thm: Wrepn.1} (2) implies that

\begin{cor}\label{cor:Wrepn.2}
Let $\chi: k^{\times}\backslash \AA_k^{\times} \rightarrow \{\pm 1\}$ be a quadratic character and $n$ is a positive odd integer. Then
$$I(0,\chi) = \left(\bigoplus_{W} R_n(W) \right)\oplus \left(\bigoplus_{\Ccal} R_n(\Ccal)\right),$$
where $W$ (resp.\ $\Ccal$) runs through all the coherent (resp.\ incoherent) quadratic spaces of dimension $n+1$ over $k$ with $\chi= \chi_W$ (resp.\ $\chi_{\Ccal}$). 
\end{cor}

\section{Siegel-Eisenstein series}\label{Sie-Eis}

Fix a quadratic character $\chi: k^{\times}\backslash \AA_k^{\times} \rightarrow \{\pm 1\}$ and a positive integer $n$.
Let $f(\cdot,s) \in I(s,\chi)$ be a \text{\it flat section}, i.e.\ for every $\kappa \in \Spp_n(O_{\AA_k})$, $f(\kappa,s)$ is independent of the chosen $s$.
The \text{\it Siegel-Eisenstein series associated with $f$} is defined by the following:
$$E(g,s,f):= \sum_{\gamma \in P_n(k)\backslash \Spp_n(k)} f(\gamma g,s), \quad \forall g \in \Spp_n(\AA_k).$$
It is known that this series converges absolutely for $\re(s)>(n+1)/2$ and can be extended to a rational function in $q^{-s}$ . Moreover, we have the following functional equation
$$E(g,s,f) = E(g,-s,M(s)(f)), \quad \forall g \in \Spp_n(\AA_k).$$
Here $M(s): I(\chi,s)\rightarrow I(\chi,-s)$ is the following intertwining operator:
$$M(s)(f)(g):= \int_{\Sym_n(\AA_k)} f(w_n \nfk(b) g,s) db, \quad w_n:= \begin{pmatrix}0&I_n\\-I_n&0\end{pmatrix},$$
and the Haar measure $db$ is \lq\lq self-dual with respect to $\psi$,\rq\rq\ i.e.\ viewing $\Sym_n(\AA_k)$ as $\AA_k^{\frac{n(n+1)}{2}}$ and write $db$ as $\prod_{1\leq i \leq j \leq n} db_{ij}$, the Haar measure $db_{ij}$ on $\AA_k$ for each pair $(i,j)$ is self-dual with respect to $\psi$.
Detecting the possible poles of $E(g,s,f)$ (cf.\ \cite{Ike}), we have that $E(g,s,f)$ is always holomorphic at the central critical value $s=0$. \\

Let $W$ be an anisotropic quadratic space over $k$ with even dimension $m$ (then $m \leq 4$). Let $\chi = \chi_W$ and $n = m-1$. For each Schwartz function $\varphi \in S(W(\AA_k)^n)$, we extend $\Phi_\varphi \in I(0,\chi)$ to be a flat section $\Phi_{\varphi}(\cdot,s) \in I(s,\chi)$ (called \text{\it the Siegel section associated with $\varphi$}) by setting
$$\Phi_{\varphi}(g,s) = |\det a|_{\AA_k}^s \Phi_{\varphi}(g)$$
for every $g = \nfk(b)\mfk(a)\kappa$ where $a \in \GL_n(\AA_k)$, $b \in \Sym_n(\AA_k)$, and $\kappa \in \Spp_n(O_{\AA_k})$.
The Siegel-Weil formula connects the central critical value $E(g,0,\Phi_{\varphi})$ with the theta series associated with the quadratic form on $W$:

\begin{thm}\label{thm: Sie-Eis.1}
\text{\rm (cf.\ \cite{Wei2})}
For every Schwartz function $\varphi \in S(W(\AA_k)^n)$ where $W$ is an anisotropic quadratic space over $k$ with even dimension $m = n+1$,
set $$\Theta(g,\varphi):= \int_{\Oo(W)(k)\backslash \Oo(W)(\AA_k)} \sum_{x \in W^n}\big(\omega(g,h)\varphi\big)(x)dh.$$
The measure $dh$ is normalized so that $\text{\rm vol}(\Oo(W)(k)\backslash \Oo(W)(\AA_k), dh) = 1$.
Then
$$E(g,0, \Phi_{\varphi}) = 2 \cdot \Theta(g,\varphi) \quad \forall g \in \Spp_n(\AA_k),$$
where $\Phi_{\varphi}(\cdot,s) \in I(s,\chi_W)$ is the Siegel section associated with $\varphi$.
\end{thm}

On the other hand, let $\Ccal$ be an incoherent quadratic space over $k$ with even dimension $m$. For every $\varphi \in S(\Ccal(\AA_k)^n)$, by the same way we can also extend $\Phi_{\varphi}$ to a flat section $\Phi_{\varphi}(\cdot,s) \in I(s,\chi_{\Ccal})$.

\begin{thm}\label{thm: Sie-Eis.2}
For every Schwartz function $\varphi \in S(\Ccal(\AA_k)^n)$ where $\Ccal$ is an incoherent quadratic space over $k$ with even dimension $m = n+1$, we have
$$E(g,0,\Phi_{\varphi}) = 0, \quad \forall g \in \Spp_n(\AA_k).$$
\end{thm}

\begin{proof}
Note that $\big(\varphi \mapsto E(\cdot ,0,\Phi_{\varphi})\big)$ is a $\Spp_n(\AA_k)$-equivariant homomorphism from $R_n(\Ccal)$ to the space of automorphic forms on $\Spp_n(\AA_k)$.
In the next section, we show that every non-singular Fourier coefficients of $E(g,0,\Phi_{\varphi})$ must be zero (see Proposition \ref{prop: Fourier.6}).
Therefore the result follows from a similar argument of \cite[Lemma 2.5]{KRS}.
\end{proof}


\section{Fourier coefficients of Siegel-Eisenstein series}\label{Fourier}

Take a positive integer $n$ and a character $\chi: k^{\times}\backslash \AA_k^{\times} \rightarrow \{\pm 1\}$.
Let $f(\cdot,s)\in I(s,\chi)$ be a flat section.
Consider the Fourier expansion of $E(g,s,f)$:
$$E(g,s,f) = \sum_{\beta \in \Sym_n(k)}E_{\beta}(g,s,f),$$
where $E_{\beta}(g,s,f)$, called the \text{\it $\beta$-th Fourier coefficient of $E(g,s,f)$}, is defined by
$$E_{\beta}(g,s,f):= \int_{\Sym_n(k)\backslash \Sym_n(\AA_k)}E(\nfk(b)g,s,f)\psi\big(-\text{Trace}(b\beta)\big)db.$$
The Haar measure $db$ is normalized so that $\text{vol}\big(\Sym_n(k)\backslash \Sym_n(\AA_k),db\big) = 1$. We can choose $db$ to be induced from the Haar measure on $\Sym_n(\AA_k)$ which is self-dual with respect to $\psi$. It is clear that
$$E_{\beta}(\nfk(b)g,s,f) =\psi_{\beta}(b) \cdot E_{\beta}(g,s,f), \quad \forall b \in \Sym_n(\AA_k),$$
where $\psi_{\beta}(b):= \psi\big(\text{Trace}(b\beta)\big)$.

\begin{lem}\label{lem: Fourier.1}
\text{\rm (cf.\ \cite[Lemma A.3]{Wei2})}
Given $\beta \in \Sym_n(k)$ with $\det \beta \neq 0$,
$$E_{\beta}(g,s,f) = \int_{\Sym_n(\AA_k)} f(w_n \nfk(b) g,s) \psi_{\beta}(-b) db.$$
\end{lem}

Note that $\Spp_n(k) P_n(\AA_k)$ is dense in $\Spp_n(\AA_k)$. For $g = \nfk(b)\mfk(a)$ where $a \in \GL_n(\AA_k)$ and $b \in \Sym_n(\AA_k)$,
$$E(g,s,f) = \sum_{\beta \in \Sym_n(k)} E_{\beta}(\mfk(a),s,f)\psi_{\beta}(b).$$
We can focus on the Fourier coefficients $E^*_{\beta}(a,s,f) := E_{\beta}(\mfk(a),s,f)$ for $a \in \GL_n(\AA_k)$ and $\beta \in \Sym_n(k)$. It is clear that
$$E_{\beta}^*(a,s,f) = E^*_{\alpha \star \beta}(\alpha^{-1} a,s,f), \quad \forall \alpha \in \GL_n(k),$$
where $\alpha \star \beta:= {}^t \alpha \beta \alpha$.


\subsection{Whittaker functions}\label{Fourier.1}

Fix a place $v$ of $k$ and a quadratic character $\chi_v:k_v^{\times}\rightarrow \{\pm 1\}$. Take a flat section $f_v(\cdot,s) \in I_v(s,\chi_v)$. For $\beta_v \in \Sym_n(k_v)$, the \text{\it local Whittaker function}
$W_{v,\beta_v}(s,f_v)$ is defined by
$$W_{v,\beta_v}(s,f_v):= \int_{\Sym_n(k_v)}f_v(w_n\nfk(b_v),s)\psi_{v,\beta_v}(-b_v)db_v.$$
Here the Haar measure $db_v$ is self-dual with respect to $\psi_v$, and $\psi_{v,\beta_v}(b_v):= \psi_v\big(\text{Trace}(b_v \beta_v)\big)$.
Let $\rho_v$ be the left action of $\Spp_n(k_v)$ on $I_v(s,\chi_v)$ by right translation. Then it is clear that
$$W_{v,\beta_v}\big(s,\rho_v(\nfk(b_v))f_v\big) = \psi_{v,\beta_v}(b_v)\cdot W_{v,\beta_v}(s,f_v), \quad \forall b_v \in \Sym_n(k_v).$$

\begin{prop}\label{prop: Fourier.2}
\text{\rm (cf.\ \cite[p.\ 102]{Ral})} Let $\beta_v \in \Sym_n(k_v)$ with $\det \beta_v \neq 0$ and take a flat section $f_v(\cdot,s) \in I_v(s,\chi_v)$.
\begin{itemize}
\item[(1)] $W_{v,\beta_v}(s,f_v)$ can be extended to an entire function on the whole $s$-plane.
\item[(2)] When $v$ is ``good", i.e.\ $\chi_v$ is unramified, $\beta_v \in \Sym_n(k_v) \cap \GL_n(O_v)$, the conductor of $\psi_v$ is trivial, and $f_v(\kappa_v) = 1$ for every $\kappa_v \in \Spp_n(O_v)$, we have
$$W_{v,\beta_v} (s,f_v) = \begin{cases} \displaystyle
L_v(s+(n+1)/2,\chi_v)^{-1} \prod_{i=1}^{(n-1)/2} \zeta_v(2s+n+1-2i)^{-1}, & \text{ if $n$ is odd,}\\
\displaystyle \prod_{i=1}^{n/2} \zeta_v(2s+n+2-2i)^{-1}, & \text{ if $n$ is even.} \end{cases}
$$
Here 
$$\zeta_v(s):= (1-q_v^{-s})^{-1} \ \text{ and } \ 
L_v(s,\chi_v):= \begin{cases} (1-\chi_v(\pi_v) q_v^{-s})^{-1} & \text{ if $\chi_v$ is unramified,} \\ 1 & \text{ otherwise.}
\end{cases}
$$
\end{itemize}
\end{prop}

Let $\chi: k^{\times} \backslash \AA_k^{\times} \rightarrow \{\pm 1\}$ be a quadratic character. 
Given $b_0 \in \Sym_n(\AA_k)$ and a flat section $f(\cdot,s) \in I(s,\chi)$,
the \text{\it global Whittaker function} $W_{b_0}(s,f)$ is defined by
$$W_{b_0}(s,f):= \int_{\Sym_n(\AA_k)} f(w_n\nfk(b),s) \psi_{b_0}(-b) db.$$
Suppose $b_0$ is also in $\GL_n(\AA_k)$ and $f$ is a pure-tensor, i.e.\ $f = \otimes_v f_v$. Let $S = S(\chi,\psi,b_0,f)$ be the finite subset of places of $k$ minimal so that $v$ is ``good" for $v \notin S$ (i.e.\ when $v \notin S$, $\chi_v$ is unramified, the conductor of $\psi_v$ is trivial, $b_{0,v} \in \Sym_n(k_v) \cap \GL_n(O_v)$, and $f_v(\kappa_v) = 1$ for every $\kappa_v \in \Spp_n(O_v)$). Define
$$\Lambda_n^S(s,\chi):= \begin{cases} \displaystyle
L^S(s+(n+1)/2,\chi)^{-1} \prod_{i=1}^{(n-1)/2} \zeta^S(2s+n+1-2i)^{-1}, & \text{ if $n$ is odd,}\\
\displaystyle \prod_{i=1}^{n/2} \zeta^S(2s+n+2-2i)^{-1}, & \text{ if $n$ is even,} \end{cases}$$
where
$$L^S(s,\chi):= \prod_{v \notin S} L_v(s,\chi_v) \ \text{ and } \ \zeta^S(s):= \prod_{v \notin S} \zeta_v(s).$$
Then $\Lambda_n^S(s,\chi)$ is holomorphic at $s=0$ and non-vanishing unless $n=1$ and $\chi \equiv 1$. Moreover,
\begin{equation}
\label{eq: Fourier.1}
W_{b_0}(s,f) = \Lambda_n^S(s,\chi)\prod_{v \in S} W_{v,b_{0,v}}(s,f_v),
\end{equation}
This gives the meromorphic continuation of the Whittaker function $W_{b_0}(s,f)$.\\

Given $\beta \in \Sym_n(k)$ and $a \in \GL_n(\AA_k)$, we set $a \star \beta:= {}^t a \beta a$.
Lemma~\ref{Fourier.1} tells us that when $\det \beta \neq 0$, the Fourier coefficient
\begin{eqnarray}
E_{\beta}^*(a,s,f) &=& \chi_{\Ccal}(\det a) |\det a|_{\AA_k}^{-s+\frac{n+1}{2}} \cdot \int_{\Sym_n(\AA_k)} f(w_n\nfk(b),s) \psi_{a \star \beta}(-b) db \nonumber \\
\label{eq: Fourier.2}
&=& \chi_{\Ccal}(\det a) |\det a|_{\AA_k}^{-s+\frac{n+1}{2}} \cdot W_{a\star \beta}(s,f), 
\end{eqnarray}
Therefore the equations (\ref{eq: Fourier.1}) and (\ref{eq: Fourier.2}) implies the following.

\begin{prop}\label{prop: Fourier.3}
Let $n$ be positive integer and $\chi: k^{\times} \backslash \AA_k^{\times} \rightarrow \{\pm 1\}$ a quadratic character. 
Given a pure-tensor flat section $f = \otimes_v f_v \in I(s,\chi)$, we have that for $\beta \in \Sym_n(k)$ with $\det \beta \neq 0$ and $a \in \GL_n(\AA_k)$,
$$\ord_{s=0}E_{\beta}^*(a,s,f) = \sum_{v \in S}\ord_{s=0} W_{v,a_v\star \beta}(s,f_v) + \epsilon_{n,\chi},$$
where $S = S(\chi,\psi, a\star \beta,f)$ is the finite subset of places of $k$ minimal so that $v$ is \lq\lq good\rq\rq\ for $v \notin S$; $\epsilon_{n,\chi} = 1$ if $n=1$ and $\chi \equiv 1$, and $\epsilon_{n,\chi} = 0$ otherwise.
\end{prop}

\subsection{The vanishing order of non-singular Fourier coefficients at $s=0$}\label{Fourier.2}

Take a place $v$ of $k$.
Let $V$ be a non-degenerate quadratic space over $k_v$ with even dimension, and take $n = \dim_k(V) - 1$.
For $\beta_v \in \Sym_n(k_v)$, we set
$$\Omega_{\beta_v}(V):= \{x \in V^n: Q_V^{(n)}(x) = \beta_v\},$$
where $Q_V^{(n)}$ is the moment map introduced in Section~\ref{Wrepn}.

\begin{lem}\label{lem: Fourier.4}
\text{\rm (cf.\ \cite[Lemma A.1]{Wei2})}
\begin{itemize}
\item[(1)] For $\beta_v \in \Sym_n(k_v)$ with $\det \beta_v \neq 0$, we have
$W_{v,\beta_v}(0,\Phi_{v,\varphi_v}) = 0$ for all $\varphi_v \in S(V^n)$ unless $\Omega_{\beta_v}(V)$ is non-empty.
\item[(2)] When $\Omega_{\beta_v}(V)$ is non-empty, $\Oo(V)$ acts on $\Omega_{\beta_v}(V)$ transitively, and there exists a constant $c$ such that for $\varphi_v \in S(V^n)$,
$$W_{v,\beta_v}(0,\Phi_{v,\varphi_v}) = c \cdot \int_{\Omega_{\beta_v}(V)}\varphi_v(x_v) dx_v,$$
where $dx_v$ is an $\Oo(V)$-invariant measure on $\Omega_{\beta_v}(V)$.
\end{itemize}
\end{lem}

The following "dichotomy" plays a fundamental role.

\begin{lem}\label{lem: Fourier.5}
\text{\rm (cf.\ \cite[Proposition 1.3]{Kud})}
For $\beta_v \in Sym_n(k_v)$ with $\det \beta_v \neq 0$, we have that $\Omega_{\beta_v}(V)$ is non-empty if and only if
\begin{equation}\label{eqn: Fourier.1}
\Hasse_v(V) = \chi_V(\det \beta_v)\cdot (\det \beta_v, -(-1)^{(n+1)/2})_v\cdot \Hasse_v(\beta_v).
\end{equation}
Here $\Hasse_v(\beta_v)$ is the Hasse invariant of the quadratic space over $k_v$ associated to $\beta_v$.
\end{lem}


Let $\Ccal$ be an incoherent quadratic space over $k$ with even dimension and $n = \dim_k(\Ccal)-1$.
For $\beta \in \Sym_n(k)$ with $\det \beta \neq 0$, We set
$$\Diff(\beta,\Ccal):= \{ \text{place $v$ of $k$} \mid \Hasse_v(\Ccal) \neq \chi_{\Ccal_v}(\det \beta)\cdot (\det \beta, -(-1)^{(n+1)/2})_v\cdot \Hasse_v(\beta)\}.$$
The incoherence of $\Ccal$ implies that the cardinality of $\Diff(\beta,\Ccal)$ must be odd. 
Moreover, by Proposition~\ref{prop: Fourier.3}, Lemma~\ref{lem: Fourier.4} and \ref{lem: Fourier.5} we have

\begin{prop}\label{prop: Fourier.6}
Let $\Ccal$ be an incoherent quadratic space over $k$ with even dimension and $n = \dim_k(\Ccal)-1$.
Given $a \in \GL_n(\AA_k)$ and $\beta \in \Sym_n(k)$ with $\det \beta \neq 0$, we have that for every Schwartz function $\varphi \in S(\Ccal(\AA_k)^n)$
$$\ord_{s=0} E_{\beta}^*(a,s,\Phi_{\varphi}) \geq \#\Diff(\beta,\Ccal) \geq 1.$$
\end{prop}

\section{Derivatives of non-singular Fourier coefficients}\label{Deriv}

Let $\Ccal$ be an incoherent quadratic space $\Ccal$ over $k$ with even dimension and $n = \dim_k(\Ccal)-1$. Given $\beta \in \Sym_n(k)$ with $\det \beta \neq 0$, by Proposition~\ref{prop: Fourier.6} we know that for $a \in \GL_n(\AA_k)$ and $\varphi \in S(\Ccal(\AA_k)^n)$,
$$\frac{\partial}{\partial s} E_{\beta}^*(a,s,\Phi_{\varphi})\bigg|_{s=0} = 0 \quad \quad \text{if $\#\Diff(\beta,\Ccal) > 1$}.$$
Suppose $\Diff(\beta,\Ccal) = \{v_0\}$. Let $V_{\beta}$ be the non-degenerate quadratic space of dimension $m$ over $k$ such that $\chi_{V_{\beta}} = \chi_{\Ccal}$ and
$$\Hasse_v(V_{\beta}) = \begin{cases} \Hasse_v(\Ccal), & \text{ if $v \neq v_0$,}\\
-\Hasse_v(\Ccal), & \text{ if $v = v_0$.} \end{cases}$$
Then for every place $v$ of $k$ different from $v_0$, we have $V_{\beta,v} \cong \Ccal_v$.
Moreover, Hasse-Minkowski principle and Lemma~\ref{lem: Fourier.5} imply that $V_{\beta}$ represents $\beta$, i.e.\ $\Omega_{\beta}(V_{\beta})$ is non-empty.

\begin{thm}\label{thm: Deriv.1}
Let $\Ccal$ be an incoherent quadratic space over $k$ with even dimension $m$ and take $n = m-1$. Given $\beta \in \Sym_n(k)$ with $\det \beta \neq 0$, suppose $\Diff(\beta,\Ccal) = \{v_0\}$ and the associated quadratic space $V_{\beta}$ is anisotropic.
Then for each pure-tensor $\varphi = \otimes_v\varphi_v \in S(\Ccal(\AA_k)^n)$ and $a \in \GL_n(\AA_k)$, we have
$$\frac{\partial}{\partial s} E_{\beta}^*(a,s,\Phi_{\varphi}) \bigg|_{s=0} = 2\cdot \frac{W_{v_0,a_{v_0}\star\beta}'(0,\Phi_{v_0,\varphi_{v_0}})}{W_{v_0,a_{v_0}\star \beta}(0,\Phi_{v_0,\widetilde{\varphi}_{v_0}})}\cdot \Theta_{\beta}^*(a,\widetilde{\varphi}),$$
where $a_{v_0}\star \beta = {}^ta_{v_0}\beta a_{v_0}$, $\widetilde{\varphi} = \otimes_v \widetilde{\varphi}_v \in S(V_{\beta}(\AA_k)^n)$ is any pure-tensor so that:
\begin{itemize}
\item[(i)] for $v \neq v_0$, $\widetilde{\varphi}_v = \varphi_v$ $($here we identify $V_{\beta,v}$ with $\Ccal_v)$;
\item[(ii)] $\widetilde{\varphi}_{v_0} \in S(V_{\beta,v_0}^n)$ is a Schwartz function satisfying that $W_{v_0,a_{v_0}\star\beta}(0,\Phi_{v_0,\widetilde{\varphi}_{v_0}}) \neq 0$;
\end{itemize}
and
$$\Theta_{\beta}^*(a,\widetilde{\varphi}) = \int_{\Sym_n(k)\backslash \Sym_n(\AA_k)} \Theta\big(\nfk(b)\mfk(a),\widetilde{\varphi}\big) \psi_{\beta}(-b)db,$$
where $\Theta(g,\widetilde{\varphi})$ is the theta series introduced in \text{\rm Theorem~\ref{thm: Sie-Eis.1}.}
\end{thm}

\begin{proof}
Take a pure-tensor $\widetilde{\varphi} = \otimes_v \widetilde{\varphi}_v \in S(V_{\beta}(\AA_k)^n)$ satisfying (i) and (ii). By the Siegel-Weil formula (Theorem~\ref{thm: Sie-Eis.1}) we have
$$E(g,0,\Phi_{\widetilde{\varphi}}) = 2 \cdot \Theta(g,\widetilde{\varphi}), \quad \forall g \in \Spp_n(\AA_k).$$
Then from the equation (\ref{eq: Fourier.2}) we obtain that for $a \in \GL_n(\AA_k)$
\begin{eqnarray}
& &\frac{\partial}{\partial s} E_{\beta}^*(a,s,\Phi_{\varphi}) \bigg|_{s=0} \nonumber \\
&=& \chi_{\Ccal}(\det a) |\det a|_{\AA_k}^{-s+\frac{n+1}{2}} W_{v_0,a_{v_0}\star \beta}'(0,\Phi_{v_0,\varphi_{v_0}}) \cdot \left(\frac{W_{a\star \beta}(0,\Phi_{\widetilde{\varphi}})}{W_{v_0,a_{v_0}\star \beta}(0,\Phi_{v_0,\widetilde{\varphi}_{v_0}})}\right) \nonumber \\
&=& \frac{W_{v_0,a_{v_0}\star \beta}'(0,\Phi_{v_0,\varphi_{v_0}})}{W_{v_0,a_{v_0}\star\beta}(0,\Phi_{v_0,\Phi_{v_0,\widetilde{\varphi}_{v_0}}})} \cdot E_{\beta}^*(a,0,\Phi_{\widetilde{\varphi}}) \nonumber \\
&=& 2\cdot \frac{W_{v_0,a_{v_0}\star \beta}'(0,\Phi_{v_0,\varphi_{v_0}})}{W_{v_0,a_{v_0}\star \beta}(0,\Phi_{v_0,\widetilde{\varphi}_{v_0}})}\cdot \Theta_{\beta}^*(a,\widetilde{\varphi}). \nonumber
\end{eqnarray}
\end{proof}

\begin{rem}
The assumption on $V_\beta$ being anisotropic is due to the restriction of Theorem~\ref{thm: Sie-Eis.1}, which can be removed as long as a stronger version of the Siegel-Weil formula is verified. 
\end{rem}



In the remaining sections, we concentrate on the special case when the incoherent quadratic space $\Ccal$ has dimension $2$, and explore the geometric interpretation of the central critical derivative of the Siegel-Eisenstein series on $\SL_2(\AA_k)$.

\section{Special case: $\dim_k(\Ccal) = 2$}\label{Examp}

Fix a place $\infty$ of $k$, referred as the place at infinity, and other places are called finite places of $k$. Let $K$ be a quadratic field over $k$ which is \lq\lq imaginary,\rq\rq\ i.e.\ $\infty$ does not split in $K$.
Take $D \in k$ such that $K = k(\sqrt{D})$. Choose $\epsilon_{\infty} \in k_{\infty}^\times$ so that the Hilbert quadratic symbol $(\epsilon_{\infty}, D)_{\infty}=-1$, and put $\epsilon_v := 1$ for every finite place $v$ of $k$. For each $\alpha \in k^\times$,
the collection $\Ccal_K^{(\alpha)} = \{\Ccal_{K,v}^{(\alpha)}\}_v$, where 
$\Ccal_{K,v}^{(\alpha)}:= (K_v, \epsilon_v \alpha N_{K/k})$, is an incoherent quadratic space over $k$ with dimension $2$. It is clear that
$$\chi_{\Ccal}(a) = \chi_K(a):= \prod_v (a_v,D)_v, \quad \forall a = (a_v)_v \in \AA_k^{\times},$$
and $\Hasse_v(\Ccal_v) = (\epsilon_v \alpha,D)_v$ for every place $v$ of $k$.
We point out that given an arbitrary incoherent quadratic space $\Ccal$ over $k$ with dimension $2$, there always exists a triple $(K,\infty,\alpha)$ such that $\Ccal \cong \Ccal_K^{(\alpha)}$.\\

Let $A$ be the ring of functions in $k$ regular away from $\infty$, and $O_K$ be the integral closure of $A$ in $K$. For each place $v$ of $k$, set $O_{K_v}:= O_K \otimes_A O_v$ when $v \neq \infty$; and $O_{K_{\infty}}$ denotes the integral closure of $O_{\infty}$ in $K_{\infty}$.
If $v$ is ramified or split in $K$, we assume that the chosen uniformizer $\pi_v$ is in $N_{K/k}(K_v)$, and pick $\Pi_v \in K_v$ such that $N_{K/k}(\Pi_v) = \pi_v$ once and for all.
For each place $v$ of $k$, set $$e_v = e_v(\alpha,\psi):= -\delta_v-\ord_v(\alpha\epsilon_v),$$
where $\delta_v$ is the \lq\lq conductor\rq\rq\ of $\psi$ at $v$ introduced in Section~\ref{Preli.bas}.
Choose a particular Schwartz function $\varphi_v^{(\alpha)} \in S(\Ccal_{K,v}^{(\alpha)})$ which is the characteristic function of the following $O_v$-lattice:
\begin{equation}\label{eqn: Examp.1}
\begin{cases}\Pi_v^{e_v}O_{K_v}, & \text{ if $v$ is ramified or split in $K$,}\\ \pi_v^{\lceil e_v/2\rceil} O_{K_v}, & \text{ if $v$ is inert in $K$.}
\end{cases} 
\end{equation}
Here $\lceil \lambda \rceil:= \text{min}\{ m \in \ZZ : m \geq \lambda\}$ for $\lambda \in \RR$.

\begin{lem}\label{lem: Examp.1}
For every $\kappa_v = \begin{pmatrix}a&b\\ c&d\end{pmatrix} \in \SL_2(O_v)$ with $c \equiv 0 \bmod \pi_v O_v$,
$$\omega_v(\kappa_v)\varphi_v^{(\alpha)} = \chi_{F,v}(d) \varphi_v^{(\alpha)}.$$
In particular, suppose \text{\rm (i)} $v$ is inert in $F$ and $e_v$ is even; or \text{\rm (ii)} $v$ splits in $F$, we get further that
$$\omega_v(\kappa_v)\varphi_v^{(\alpha)} = \varphi_v^{(\alpha)}, \quad \forall \kappa_v \in \SL_2(O_v).$$
\end{lem}

\begin{proof}
One observes that the Fourier transform of $\varphi_v^{(\alpha)}$ is 
$$\widehat{\varphi}_v^{(\alpha)} = \begin{cases}
\mathbf{1}_{\Pi_v^{e_v}O_{F_v}}, & \text{ if $v$ is split in $k$,}\\
q^{-1/2} \cdot \mathbf{1}_{\Pi_v^{e_v-1}O_{F_v}}, & \text{ if $v$ is ramified in $k$,} \\
q^{((-1)^{e_v}-1)/2} \cdot \mathbf{1}_{\pi_v^{\lfloor e_v/2 \rfloor}O_{F_v}}, & \text{ if $v$ is inert in $k$.}
\end{cases}$$
Here $\lfloor \lambda \rfloor := \text{max}\{ m \in \ZZ: m \leq \lambda\}$ for $\lambda \in \RR$.
By the decomposition
$$\begin{pmatrix}a&b\\ \pi_v c & d\end{pmatrix} = \begin{pmatrix}1&bd^{-1}\\0&1\end{pmatrix}\begin{pmatrix}d^{-1}&0\\0&d\end{pmatrix} \begin{pmatrix}0&1\\ -1&0\end{pmatrix} \begin{pmatrix}1&- \pi_v d^{-1}c \\ 0 &1\end{pmatrix} \begin{pmatrix}0&-1\\ 1 & 0\end{pmatrix} \in \SL_2(O_v),$$
the result is then straightforward.
\end{proof}
 
Let $\varphi^{(\alpha)}$ be the pure-tensor $\otimes_v \varphi_v^{(\alpha)} \in S\big(\Ccal_K^{(\alpha)}(\AA_k)\big)$. By the above lemma, $\Phi_{\varphi^{(\alpha)}}$ is viewed as a \lq\lq new vector\rq\rq\ in the representation $R_1(\Ccal_K^{(\alpha)})$ of $\SL_2(\AA_k)$.
Denote by $\FF_K$ and $g_K$ the constant field and the genus of $K$ respectively. Set
$$L(s,\chi_K):= \prod_v L_v(s,\chi_{K,v}) \quad \text{ and } \quad \widetilde{L}(s,\chi_K):= q^{\big([\FF_K:\FF_q](g_K-1)-(g_k-1)\big)s}L(s,\chi_K),$$ 
which are extended to rational functions in $q^{-s}$ with the functional equation 
$$\widetilde{L}(s,\chi_K) = \widetilde{L}(1-s,\chi_K).$$
Consider the following (modified) Eisenstein series:
\begin{equation}\label{eqn: Examp.2}
\widetilde{E}(g,s,\Phi_{\varphi^{(\alpha)}}):= \widetilde{L}(s+1,\chi_F)\cdot E(g,s,\Phi_{\varphi^{(\alpha)}}).
\end{equation}
It is observed that for $g \in \SL_2(\AA_k)$, $\widetilde{E}(g,s,\Phi_{\varphi^{(\alpha)}})$ is a rational function in $q^{-s}$ satisfying:
$$\widetilde{E}(g,s,\Phi_{\varphi^{(\alpha)}}) = -  \widetilde{E}(g,-s,\Phi_{\varphi^{(\alpha)}})\cdot \left(\prod_{v \text{ is inert} \atop \text{and $e_v$ is odd}} q_v^{-s}\frac{L_v(1+s,\chi_{F,v})}{L_v(1-s,\chi_{F,v})}\right) .$$

We are interested in its central critical derivative
$$\eta^{(\alpha)}(g) := \frac{\partial}{\partial s} \widetilde{E}(g,s,\Phi_{\varphi^{(\alpha)}})\Big|_{s=0}, \quad \forall g \in \SL_2(\AA_k).$$
Let $\Nfk=\Nfk(\psi,K,\alpha)$ be the positive divisor of $k$ such that 
$$\ord_v(\Nfk) = \begin{cases} 1, & \text{ if either $v$ is ramified in $k$ or $v$ is inert in $k$ and $e_v$ is odd,} \\
0, & \text{ otherwise.}\end{cases}$$
Lemma~\ref{lem: Examp.1} implies that $\eta^{(\alpha)}$ is uniquely determined by its values on the representatives of the double cosets in $\SL_2(k) \backslash \SL_2(\AA_k) / \Kcal_0(\Nfk)$, where 
$$\Kcal_0(\Nfk) = \left\{\begin{pmatrix}a&b\\ c&d\end{pmatrix} \in \SL_2(O_{\AA_k}) \ \bigg| \ c_v \equiv 0 \bmod \pi_v^{\ord_v(\Nfk)}O_v \text{ for all } v \right\}.$$
Let $B$ be the standard Borel subgroup of $\SL_2$, i.e.\
$$B = P_1 = \{\nfk(x)\mfk(y)\mid x \in \GG_a,\ y \in \GL_1\}.$$
The canonical map from $B(\AA_k)$ to the double coset space $\SL_2(k)\backslash \SL_2(\AA_k) / \Kcal_0(\Nfk)$ is surjective.
Therefore to explore the derivative $\eta^{(\alpha)}$, it suffices to compute the Fourier coefficients $\eta^{(\alpha),*}_{\beta}(y)$ for $y \in \AA_k^{\times}$ and $\beta \in k$.

\subsection{The constant term $\eta^{(\alpha),*}_0$}

We first compute the constant term $\eta^{(\alpha),*}_0$ in the following:

\begin{lem}\label{lem: Examp.2}
For $y \in \AA_k^{\times}$, we have
\begin{eqnarray}
\eta^{(\alpha),*}_{0}(y) &=& 2\chi_K(y)|y|_{\AA_k}L(0,\chi_K) \nonumber \\
&& \!\!\!\!\!\!\!\!\!\!\!\!\!\!\! \cdot \left[ \ln|y|_{\AA_k}-\Big([\FF_K:\FF_q](g_K-1)-(g_k-1)\Big)\ln q - \frac{L'(0,\chi_K)}{L(0,\chi_K)}
+\frac{1}{2}\sum_{v \text{ \rm  is inert} \atop \text{\rm and $e_v$ is odd}}\frac{q_v-1}{q_v+1}\ln q_v\right]. \nonumber
\end{eqnarray}
\end{lem}

\begin{proof}
For $y \in \AA_k^{\times}$, we claim that the constant term $\widetilde{E}_0^*(y,s,\Phi_{\varphi^{(\alpha)}})$ is equal to

\begin{equation}
\label{eq: Examp.1}
\chi_K(y)|y|_{\AA_k} \left(|y|_{\AA_k}^s \widetilde{L}(-s,\chi_K) - |y|_{\AA_k}^{-s}\widetilde{L}(s,\chi_K) \cdot \prod_{v \text{ is inert} \atop \text{and $e_v$ is odd}}q_v^{-s}\frac{L_v(1+s,\chi_{K,v})}{L_v(1-s,\chi_{K,v})}\right).
\end{equation}
Then the result follows.

To prove the equality (\ref{eq: Examp.1}), we recall the fact that
$$E_0^*(y,s,\Phi_{\varphi^{(\alpha)}}) = \Phi_{\varphi^{(\alpha)}}(\mfk(y),s) + 
\int_{\AA_k} \Phi_{\varphi^{(\alpha)}}\left(\begin{pmatrix}0&1\\-1&0\end{pmatrix} \begin{pmatrix}1&b\\0&1\end{pmatrix}\begin{pmatrix}y&0\\0&y^{-1}\end{pmatrix},s\right) db.$$
The Haar measure $db$ is normalized to be self-dual with respect to $\psi$.
More precisely, if we write $db$ as $\prod_v db_v$, the Haar measure $db_v$ on $k_v$ for each place $v$ is choosed so that $\text{vol}(O_v,db_v) = q_v^{-\delta_v/2}$.
One observes that $\Phi_{\varphi^{(\alpha)}}(\mfk(y),s) = \chi_K(y)|y|_{\AA_k}^{s+1}$, and for each place $v$ of $k$, the local integral
$$\int_{k_v} \Phi_{v,\varphi_v^{(\alpha)}}\left(\begin{pmatrix}0&1\\-1&0\end{pmatrix} \begin{pmatrix}1&b_v\\0&1\end{pmatrix}\begin{pmatrix}y_v&0\\0&y_v^{-1}\end{pmatrix},s\right) db_v$$
is equal to
$$q_v^{-\delta_v/2}\chi_{K,v}(y_v)|y_v|_v^{-s+1}\varepsilon_v(\Ccal_{K,v}^{(\alpha)})\cdot 
\extrarowheight=8pt
\begin{cases}
 q_v^{-1/2}, & \text{ if $v$ is ramified in $K$,} \\
\frac{\displaystyle L_v(s,\chi_{K,v})}{\displaystyle L_v(s+1,\chi_{K,v})}, & \text{ if $v$ splits in $K$,} \\
q_v^{\frac{(-1)^{e_v}-1}{2}}\frac{\displaystyle L_v(s,\chi_{K,v})}{\displaystyle L_v(s+(-1)^{e_v},\chi_{K,v})}, & \text{ if $v$ is inert in $K$.}
\end{cases}$$
Since $\Ccal_K^{(\alpha)}$ is incoherent, Weil's reciprocity (cf.\ \cite{We2}) implies that 
$$\prod_v \varepsilon(\Ccal_v) = -1.$$
Recall that $\sum_v \delta_v \deg v = 2g_k-2$, where $g_k$ is the genus of $k$. By Hurwitz formula we get
$$(g_k-1)+ \sum_{v \text{ is ramfied in $K$}} \deg v/2 = [\FF_K:\FF_q](g_K-1)-(g_k-1).$$
Hence the constant term $E_0^*(y,s,\Phi_{\varphi^{(\alpha)}})$ is equal to
$$\chi_K(y)|y|_{\AA_k}^{s+1}-\chi_K(y)|y|_{\AA_k}^{-s+1} q^{-[\FF_K:\FF_q](g_K-1)+(g_k-1)}\frac{L(s,\chi_K)}{L(s+1,\chi_K)} \cdot \prod_{v \text{ is inert in $K$} \atop \text{and $e_v$ is odd}} q_v^{-1}\frac{L_v(s+1,\chi_K)}{L_v(s-1,\chi_K)}.$$
The equality~(\ref{eq: Examp.1}) holds immediately.
\end{proof}

\begin{rem}\label{rem: Examp.3}
1. The \lq\lq symmetry\rq\rq\ of the constant term described in the equality (\ref{eq: Examp.1}) agrees with the functional equation of $\widetilde{E}(g,s,\Phi_{\varphi^{(\alpha)}})$. 
In particular, when $e_v$ is even for every inert place $v$ and $K/k$ is not a constant field extension, $\widetilde{E}(g,s,\Phi_{\varphi^{(\alpha)}})$ is actually a polynomial in $\CC[q^s,q^{-s}]$ satisfying
$$\widetilde{E}(g,s,\Phi_{\varphi^{(\alpha)}})= - \widetilde{E}(g,-s,\Phi_{\varphi^{(\alpha)}}).$$
2. Let $f_{\infty}$ be the residue degree of $\infty$ in $K/k$. The following equality
$$L(0,\chi_K) = \frac{\#\Pic(O_K)}{f_{\infty} \cdot \#\Pic(A)}$$
will be used later.
\end{rem}

\subsection{Non-constant Fourier coefficient $\eta^{(\alpha),*}_\beta$}

Given $\beta \in k^{\times}$, we know that (by Proposition~\ref{prop: Fourier.6}) when $\#\Diff(\beta,\Ccal_K^{(\alpha)})>1$, 
$$\eta^{(\alpha),*}_{\beta}(y) = 0 \quad \forall y \in \AA_k^{\times}.$$
Suppose $\Diff(\beta,\Ccal_K^{(\alpha)}) = \{v_0\}$.
Then $v_0$ must be ramified or inert in $K$.
The quadratic space $V_{\beta} := (K,\beta N_{K/k})$ over $k$ represents $\beta$, i.e.\
$\Omega_{\beta}(V_{\beta})$ is non-empty. Moreover, $\chi_{V_{\beta}} = \chi_K$ and
$$\Hasse_v(V_{\beta,v}) = \begin{cases} \Hasse_v(\Ccal_{K,v}^{(\alpha)}), & \text{ if $v \neq v_0$,}\\
-\Hasse_v(\Ccal_{K,v}^{(\alpha)}), & \text{ if $v = v_0$.}\end{cases}$$
For each place $v$ of $k$, let $e_v'= -\delta_v - \ord_v(\beta)$.
Take $\widetilde{\varphi}^{(\beta)} = \otimes_v\widetilde{\varphi}^{(\beta)}_v \in S(V_{\beta}(\AA_k))$ where for every place $v$ of $k$,
$\widetilde{\varphi}^{(\beta)}_v$ is the characteristic function of the $O_v$-lattice 
$$\begin{cases}\Pi_v^{e_v'}O_{F_v}, & \text{ if $v$ is ramified or split in $F$,}\\ 
\pi_v^{\lceil e_v'/2\rceil} O_{F_v}, & \text{ if $v$ is inert in $F$.}
\end{cases} $$
For $v \neq v_0$, one observes that $\widetilde{\varphi}^{(\beta)}_v = \varphi_v^{(\alpha)}$ (when identifying $V_{\beta,v}$ with $\Ccal_{K,v}^{(\alpha)}$).
To know $\eta^{(\alpha),*}_{\beta}(y)$, by Theorem~\ref{thm: Deriv.1} it suffices to compute
$W_{v_0,y_{v_0}^2\beta}(0,\Phi_{v_0,\widetilde{\varphi}_{v_0}^{(\beta)}})$,
$W_{v_0,y_{v_0}^2\beta}'(0,\Phi_{v_0,\varphi_{v_0}^{(\alpha)}})$, and
$\Theta_{\beta}^*(y,\widetilde{\varphi}^{(\beta)})$.\\

Note that
\begin{eqnarray}
\label{eq: Examp.2}
W_{v_0,y_{v_0}^2\beta}(s,\Phi_{v_0,\widetilde{\varphi}_{v_0}^{(\beta)}})
&= & \int_{k_{v_0}}\Phi_{v_0,\widetilde{\varphi}_{v_0}^{(\beta)}}
\left(\begin{pmatrix}0&1\\-1&0\end{pmatrix}
\begin{pmatrix} 1&b\\0&1\end{pmatrix},s\right) \psi_{v_0}(-y_{v_0}^2\beta b)db \nonumber \\
&= &\varepsilon_{v_0}(V_{\beta,v_0})\widehat{\widetilde{\varphi}_{v_0}^{(\beta)}}(0) \cdot \int_{O_{v_0}}\psi_v(-y_{v_0}^2\beta b) db \\
&& + \sum_{r=1}^{\infty} \big(\chi_{K,v_0}(\pi_{v_0})q_{v_0}^{-s}\big)^r \cdot \int_{O_{v_0}^{\times}} \chi_{K,v_0}(u) \psi_v(y_{v_0}^2\beta \pi_{v_0}^{-r} u) d^{\times} u. \nonumber 
\end{eqnarray}
It is clear that
$$\int_{O_{v_0}}\psi_v(-y_{v_0}^2\beta b) db = 
\begin{cases}
\text{vol}(O_{v_0}), & \text{ if $\ord_{v_0}(y_{v_0}^2 \beta)+\delta_{v_0} \geq 0$, } \\
0, & \text{ otherwise.}
\end{cases}$$

\subsubsection{The case when $v_0$ is inert in $K$}

One has
$$\int_{O_{v_0}^{\times}}\psi_v(y_{v_0}^2\beta \pi_{v_0}^{-r} u) d^{\times} u = \text{vol}(O_v)\cdot 
\begin{cases} (1-q_{v_0}^{-1}), & \text{ if $r \leq \ord_{v_0}(y_{v_0}^2 \beta)+\delta_{v_0}$,} \\
-q_{v_0}^{-1}, & \text{ if $r = \ord_{v_0}(y_{v_0}^2 \beta)+\delta_{v_0}+1$,}\\
0, & \text{ if $r > \ord_{v_0}(y_{v_0}^2 \beta)+\delta_{v_0}+1$.}
\end{cases}$$
Since
$$\varepsilon_{v_0}(V_{\beta,v_0})\widehat{\widetilde{\varphi}_{v_0}^{(\beta)}}(0) 
= \begin{cases}
1, & \text{ if $e_v'$ is even,} \\
-q_v^{-1}, & \text{ if $e_v'$ is odd,}
\end{cases}$$
we obtain the following result.

\begin{lem}\label{lem: Examp.4}
Suppose $v_0$ is inert in $K$. $W_{v_0,y_{v_0}^2\beta}(s,\Phi_{v_0,\widetilde{\varphi}_{v_0}^{(\beta)}}) = 0$ if $\ord_{v_0}(y_{v_0}^2 \beta)+\delta_{v_0}<0$. When $\ord_{v_0}(y_{v_0}^2 \beta)+\delta_{v_0}\geq 0$, $W_{v_0,y_{v_0}^2\beta}(s,\Phi_{v_0,\widetilde{\varphi}_{v_0}^{(\beta)}})$ is equal to
$$\text{\rm vol}(O_{v_0}) \cdot
\extrarowheight=10pt
\begin{cases}
\left(1-(-q_{v_0}^{-s})^{\ord_{v_0}(y_{v_0}^2 \beta)+\delta_{v_0}+1}\right) \frac{\displaystyle L_{v_0}(s,\chi_K)}{\displaystyle L_{v_0}(s+1,\chi_K)}, & \text{ if $e_{v_0}'$ is even,}\\
- \frac{\displaystyle q_{v_0}^{-s}\left(1-(-q_{v_0}^{-s})^{\ord_{v_0}(y_{v_0}^2 \beta)+\delta_{v_0}}\right)+q_{v_0}^{-1}\left(1-(-q_{v_0}^{-s})^{\ord_{v_0}(y_{v_0}^2 \beta)+\delta_{v_0}+2}\right)}{\displaystyle 1+q_{v_0}^{-s}}, & \text{ if $e_{v_0}'$ is odd.}
\end{cases}$$
In particular, when $\ord_{v_0}(y_{v_0}^2 \beta)+\delta_{v_0}\geq 0$, 
$W_{v_0,y_{v_0}^2\beta}(0,\Phi_{v_0,\widetilde{\varphi}_{v_0}^{(\beta)}}) = q_{v_0}^{-\delta_{v_0}/2}(-1)^{e_{v_0}'}(1+q_{v_0}^{-1})$.
\end{lem}

When $v_0$ is inert in $K$, by our assumption we have 
$$(\alpha\epsilon_{v_0},D)_{v_0} = - (\beta,D)_{v_0},$$
which means that the parity of $e_{v_0}$ and $e_{v_0}'$ must be different. Similarly, we can get:

\begin{lem}\label{lem: Examp.5}
Suppose $v_0$ is inert in $K$, 
$W_{v_0,y_{v_0}^2\beta}(s,\Phi_{v_0,\varphi_{v_0}^{(\alpha)}}) = 0$ if $\ord_{v_0}(y_{v_0}^2 \beta)+\delta_{v_0}<0$. When $\ord_{v_0}(y_{v_0}^2 \beta)+\delta_{v_0}\geq 0$, 
$W_{v_0,y_{v_0}^2\beta}(s,\Phi_{v_0,\varphi_{v_0}^{(\alpha)}})$ is equal to
$$\text{\rm vol}(O_{v_0}) \cdot
\extrarowheight=10pt
\begin{cases}
\left(1-(-q_{v_0}^{-s})^{\ord_{v_0}(y_{v_0}^2 \beta)+\delta_{v_0}+1}\right) \frac{\displaystyle L_{v_0}(s,\chi_K)}{\displaystyle L_{v_0}(s+1,\chi_K)}, & \text{ if $e_{v_0}'$ is odd,}\\
- \frac{\displaystyle q_{v_0}^{-s}\left(1-(-q_{v_0}^{-s})^{\ord_{v_0}(y_{v_0}^2 \beta)+\delta_{v_0}}\right)+q_{v_0}^{-1}\left(1-(-q_{v_0}^{-s})^{\ord_{v_0}(y_{v_0}^2 \beta)+\delta_{v_0}+2}\right)}{\displaystyle 1+q_{v_0}^{-s}}, & \text{ if $e_{v_0}'$ is even.}
\end{cases}$$
In particular, when $\ord_{v_0}(y_{v_0}^2 \beta)+\delta_{v_0}\geq 0$, 
\begin{eqnarray}
W_{v_0,y_{v_0}^2\alpha}'(0,\Phi_{v_0,\varphi_{v_0}^{(\alpha)}}) 
&=& q_{v_0}^{-\delta_{v_0}/2} \frac{(-1)^{(e_{v_0}'-1)}}{2} \ln q_{v_0} \cdot \bigg[\Big(\ord_{v_0}(y_{v_0}^2\beta)+\delta_{v_0}+1-\frac{1+(-1)^{e_{v_0}'}}{2}\Big)\nonumber \\
& & \quad \quad \quad \quad +\ q_{v_0}^{-1}\Big(\ord_{v_0}(y_{v_0}^2\beta)+\delta_{v_0}+1+\frac{1+(-1)^{e_{v_0}'}}{2}\Big) \bigg]. \nonumber
\end{eqnarray}
\end{lem}

\subsubsection{The case when $v_0$ is ramified in $K$}
By our assumption we have $\chi_{K,v_0}(\pi_{v_0}) = 1$. 
One observes that
$$\int_{O_{v_0}^{\times}}\chi_{K,v_0}(u)\psi_v(y_{v_0}^2\beta \pi_{v_0}^{-r} u) d^{\times} u
= \text{vol}(O_{v_0})\cdot \begin{cases}
\varepsilon_{v_0}(V_{\beta,v_0})q_{v_0}^{-1/2}, & \text{ if $ r = \ord_{v_0}(y_{v_0}^2\beta)+\delta_{v_0}+1$,} \\
0,& \text{ otherwise.}
\end{cases}$$
Moreover, $\widehat{\widetilde{\varphi}_{v_0}^{(\beta)}}(0) = q_{v_0}^{-1/2}$. Therefore from the equation~(\ref{eq: Examp.2}) we get:

\begin{lem}\label{lem: Examp.6}
Suppose $v_0$ is ramified in $F$, 
$W_{v_0,y_{v_0}^2\beta}(s,\Phi_{v_0,\widetilde{\varphi}_{v_0}^{(\beta)}}) = 0$ if $\ord_{v_0}(y_{v_0}^2 \beta)+\delta_{v_0}<0$. When $\ord_{v_0}(y_{v_0}^2 \beta)+\delta_{v_0}\geq 0$, 
$W_{v_0,y_{v_0}^2\beta}(s,\Phi_{v_0,\widetilde{\varphi}_{v_0}^{(\beta)}})$ is equal to
$$\text{\rm vol}(O_{v_0}) \cdot \varepsilon_{v_0}(V_{\beta,v_0}) q_{v_0}^{-1/2} ( 1+ q_{v_0}^{-(\ord_{v_0}(y_{v_0}^2 \beta)+\delta_{v_0}+1)s}).$$
In particular, $W_{v_0,y_{v_0}^2\beta}(0,\Phi_{v_0,\widehat{\varphi}_{v_0}^{(\beta)}}) = 2 \varepsilon_{v_0}(V_{\beta,v_0}) q_{v_0}^{-(\delta_{v_0}+1)/2}$.
\end{lem}

Since $\varepsilon_{v_0}(\Ccal_{K,v_0}^{(\alpha)}) = -\varepsilon_{v_0}(V_{\beta,v_0})$, we obtain that:

\begin{lem}\label{lem: Examp.7}
Suppose $v_0$ is ramified in $K$, 
$W_{v_0,y_{v_0}^2\beta}(s,\Phi_{v_0,\varphi_{v_0}^{(\alpha)}}) = 0$ if $\ord_{v_0}(y_{v_0}^2 \beta)+\delta_{v_0}<0$. When $\ord_{v_0}(y_{v_0}^2 \beta)+\delta_{v_0}\geq 0$, 
$W_{v_0,y_{v_0}^2\beta}(s,\Phi_{v_0,\varphi_{v_0}^{(\alpha)}})$ is equal to
$$-\text{\rm vol}(O_{v_0}) \cdot \varepsilon_{v_0}(V_{\beta,v_0}) q_{v_0}^{-1/2} ( 1- q_{v_0}^{-(\ord_{v_0}(y_{v_0}^2 \beta)+\delta_{v_0}+1)s}).$$
In particular, $W_{v_0,y_{v_0}^2\beta}'(0,\Phi_{v_0,\varphi_{v_0}^{(\alpha)}}) = - \varepsilon_{v_0}(V_{\beta,v_0}) q_{v_0}^{-(\delta_{v_0}+1)/2} \ln q_{v_0} (\ord_{v_0}(y_{v_0}^2 \beta)+\delta_{v_0}+1)$.
\end{lem}

\subsubsection{Computing $\Theta_{\beta}^*(y,\widetilde{\varphi}^{(\beta)})$}
Note that
$$\Theta_{\beta}^*(y,\widetilde{\varphi}^{(\beta)})
= \chi_F(y)|y|_{\AA_k} \cdot \int_{\Oo(V_{\beta})(k)\backslash \Oo(V_{\beta})(\AA_k)} \sum_{x \in F, \atop N_{F/k}(x) = 1} \widetilde{\varphi}^{(\beta)}(h^{-1}xy) dh.$$
For each place $v$ of $k$, we identify $\Oo(V_{\beta})(k_v)$ with $K_v^{1} \times \langle\tau_v \rangle$, where $\tau_v$ is the non-trivial $k_v$-automorphism on $K_v$ and $K_v^1 = \{ x \in K_v : N_{K/k}(x) = 1\}$.
By Hilbert's theorem $90$, the map $\big(a \mapsto a/\tau_v(a)\big): K_v^{\times} \rightarrow K_v^1$ is surjective. It is clear that $\widetilde{\varphi}_v^{(\beta)}(h^{-1}xy) = \widetilde{\varphi}_v^{(\beta)}(xy)$ when $h$ is of the form $u/\tau_v(u)$ for $u \in O_v^{\times}$. Moreover, when $v$ does not split in $K$, $\widetilde{\varphi}_v^{(\beta)}$ is fixed by the action of $\tau_v$. Hence we can express $\Theta_{\beta}^*(y,\widetilde{\varphi}^{(\beta)})$ as follows.

\begin{lem}\label{lem: Examp.8}
Suppose $\ord_v(y_v^2\beta)+\delta_v \geq 0$ for every place $v$ of $k$.
Then 
$$\Theta_{\beta}^*(y,\widetilde{\varphi}^{(\beta)})
= \frac{\chi_K(y)|y|_{\AA_k}}{\#\Pic(O_K)} \sum_{[\Afk] \in \Pic(O_K)} \#\{ x \in \Afk \bar{\Afk}^{-1} \yfk^{-1}\Dfk_{\beta}^{-1}  : N_{K/k}(x) = 1\}.$$
Here $\yfk$ is the ideal of $O_K$ such that $\yfk_v = y_v O_{K_v}$ for every finite place $v$ of $k$;
and $\Dfk_{\beta}$ is the ideal of $O_K$ such that for each finite place $v$ of $k$,
$$\Dfk_{\beta,v} = \Dfk_{\beta} \otimes_A O_v = \begin{cases} \Pi_v^{\delta_v+\ord_v(\beta)} O_{K_v}, & \text{ if $v$ is ramified or split in $K$, } \\
\pi_v^{\lfloor \frac{\delta_v + \ord_v(\beta)}{2}\rfloor} O_{K_v}, & \text{ if $v$ is inert in $K$.}\end{cases}$$
\end{lem}

We summarize what we got so far in the following.

\begin{prop}\label{prop: Examp.9}
Let $\beta \in k^{\times}$ with $\Diff(\beta,\Ccal_K^{(\alpha)}) = \{v_0\}$. For $y \in \AA_k^{\times}$,
$\eta^{(\alpha)}_{\beta}(y) = 0$ if there exists a place $v$ such that $\ord_v(y_v^2\beta)+\delta_v <0$.
Suppose $\ord_v(y_v^2\beta)+\delta_v \geq 0$ for every place $v$ of $k$.
\begin{itemize}
\item[(1)] When $v_0$ is inert in $K$, $\eta^{(\alpha),*}_{\beta}(y)$ is equal to
\begin{eqnarray}
&& -\frac{\chi_K(y)|y|_{\AA_k}}{f_{\infty} \#\Pic(A)} \left(\frac{\ln q_{v_0}}{1+q_{v_0}}\right) \nonumber \\
&& \cdot \bigg[q_{v_0}\Big(\ord_{v_0}(y_{v_0}^2\beta)+\delta_{v_0}+1-\frac{1+(-1)^{e_{v_0}'}}{2}\Big)
 +\Big(\ord_{v_0}(y_{v_0}^2\beta)+\delta_{v_0}+1+\frac{1+(-1)^{e_{v_0}'}}{2}\Big) \bigg] \nonumber \\
&&\cdot \sum_{[\Afk] \in \Pic(O_K)} \#\{ x \in \Afk \bar{\Afk}^{-1} \yfk^{-1}\Dfk_{\beta}^{-1}  : N_{K/k}(x) = 1\}.
\nonumber
\end{eqnarray}
\item[(2)] When $v_0$ is ramified in $K$, $\eta^{(\alpha),*}_{\beta}(y)$ is equal to
$$-\frac{\chi_K(y)|y|_{\AA_k}\ln q_{v_0}}{f_{\infty} \#\Pic(A)} \cdot \big(\ord_{v_0}(y_{v_0}^2\beta)+\delta_{v_0}+1\big) 
\cdot \sum_{[\Afk] \in \Pic(O_K)} \#\{ x \in \Afk \bar{\Afk}^{-1} \yfk^{-1}\Dfk_{\beta}^{-1}  : N_{K/k}(x) = 1\}.$$
\end{itemize}
Here $\yfk$ and $\Dfk_\beta$ are the ideals of $O_K$ introduced in \text{\rm Lemma~\ref{lem: Examp.8}}.
\end{prop}

\begin{rem}\label{rem: Examp.10}
For our purpose in the next section, we shall express the counting number 
$\#\{ x \in \Afk \bar{\Afk}^{-1} \yfk^{-1}\Dfk_{\beta}^{-1}  : N_{K/k}(x) = 1\}$ by a different form.
Given an arbitrary $\gamma \in k^{\times}$,
consider the quaternion algebra $\Dcal_{\gamma} := K + K j_{\gamma}$ over $k$, where $j_{\gamma}^2 = -\gamma$ and $j_{\gamma} a = \bar{a} j_{\gamma}$ for every $a \in K$. Then $\Dcal_{\gamma}$ splits at $v$ if and only if $\chi_{K,v}(-\gamma) = 1$.
Moreover, the reduced norm form on $\Dcal_{\gamma}$ is $N_{K/k} \oplus (\gamma K_{F/k})$.
For each finite place $v$ of $k$ which is ramified in $K$, let $\Pfk_v$ be the prime ideal of $O_K$ lying above $v$.
Set
$$\dfk_{K,\gamma}:= \prod_{v \neq \infty, \text{$v$ is ramified in $K$,} \atop \chi_{K,v}(-\gamma) = 1} \Pfk_v.$$
Take an ideal $\Cfk$ of $O_K$ such that for each finite place $v$ of $k$,
$$\ord_v\big(\text{N}_{K/k}(\Cfk)\big) =
\begin{cases} 
\ord_v(\gamma), & \text{ if $v$ is ramified or split in $K$,} \\
2\cdot \lfloor\frac{\ord_v(\gamma)}{2}\rfloor, & \text{ if $v$ is inert in $K$.}
\end{cases}$$
For each place $v$ of $k$ where $\chi_{K,v}(-\gamma) = 1$, take $\xi_v \in K_v^{\times}$ such that $\text{N}_{K/k}(\xi_v) = -\gamma$. Set
$$R^{(\gamma)}:=\{a+bj_{\gamma} \mid a \in \dfk_{K,\gamma}^{-1}, \ b \in \dfk_{K,\gamma}^{-1}\Cfk^{-1}, \
a \equiv \xi_v b \bmod O_{K_v} \ \forall v \text{ ramified in $K$ and $\chi_{K,v}(-\gamma) = 1$} \}.$$
It is observed that $R^{(\gamma)}$ is a maximal $A$-order in $\Dcal_{\gamma}$ (by computing the discriminant of $R^{(\gamma)}$ over $A$).\\

(1) Given $\beta \in k^{\times}$ with $\Diff(\beta,\Ccal_K^{(\alpha)}) = \{v_0\}$ and $v_0 \neq \infty$, take $\gamma = \gamma(\alpha,\beta) = -\beta/\alpha$. Then $\Dcal_{\gamma}$ is ramified precisely at $v_0$ and $\infty$. Moreover, we have the following isomorphism (of central simple algebras over $k$):
\begin{equation}
\label{eqn: Examp.5}
\begin{tabular}{ccc}
$\Dcal_{\beta} \otimes_k \Dcal_{\alpha}$ & $\cong$ & $\Mat_2(\Dcal_{\gamma})$ \\
$1 \otimes (a_1 + a_2 j_{\alpha})$ & $\longmapsto$ & $\begin{pmatrix}a_1 & a_2 \\ -\alpha \overline{a_2} & \overline{a_1} \end{pmatrix}$\\
$(a_3+a_4 j_{\beta})\otimes 1$ & $\longmapsto$ & $\begin{pmatrix}a_3& a_4 j_{\gamma} \\ -\alpha a_4 j_{\gamma}& a_3\end{pmatrix}$ \\
\end{tabular}
\end{equation}
Viewing $\Dcal_{\alpha}$ and $\Dcal_{\beta}$ as subalgebras in $\Mat_2(\Dcal_{\gamma})$ via the above isomorphism, let $$\widetilde{R}:= \{ M \in \Mat_2(R^{(\gamma)}) : M u = u M,\ \forall u \in \Dcal_{\alpha}\} = \Mat_2(R^{(\gamma)}) \cap \Dcal_{\beta} \subset \Dcal_{\beta}.$$
Set $\widetilde{R}_{\Afk}:= \Afk \widetilde{R} \Afk^{-1}$ for each fractional ideal $\Afk$ of $O_K$. 

Without loss of generality, assume $\alpha \in A$. Choose a suitable additive character $\psi$ so that $\ord_v(\alpha)+\delta_v$ is even for every place $v$ of $k$ inert in $K$. 
Then sending $x$ to $x j_\beta$ for $x \in K$ induces a bijection between $\{x \in \Afk \bar{\Afk}^{-1} \yfk^{-1}\Dfk_{\beta}^{-1}  : N_{K/k}(x) = 1\}$ and
$$
\{M \in \widetilde{R}_{\Afk}\yfk^{-1} \overline{\Dfk}_{\alpha}^{-1} : M u = u M, \ \forall u \in \Dcal_\alpha; \ M \cdot a  = \overline{a} \cdot M, \ \forall a \in K; \text{ and } M^2 = -\beta\}.$$
Here we embed $K \hookrightarrow \Dcal_\gamma \hookrightarrow \Mat_2(\Dcal_\gamma)$.
${}$\\
(2) When $\Diff(\beta,\Ccal_K^{(\alpha)}) = \{\infty\}$, we have $\beta/\alpha \in N_{K/k}(K^{\times})$ and a bijection
$$\{x \in \Afk \bar{\Afk}^{-1} \yfk^{-1}\Dfk_{\beta}^{-1}  : N_{K/k}(x) = 1\} \quad \cong \quad \{x \in \Afk \bar{\Afk}^{-1}\yfk^{-1} \Dfk_{\alpha}^{-1} : N_{K/k} = \beta/\alpha \}.$$
\end{rem}

\section{Algebro-geometric aspects}\label{Alg-geo}

In this section, we assume (without loss of generality) $\alpha \in A$ with $\ord_v(\alpha) + \delta_v$ is even for every place $v$ of $k$ inert in $K$, and connect the non-zero Fourier coefficients of $\eta^{(\alpha)}$ with the degree of special cycles on the coarse moduli scheme of rank one Drinfeld $O_K$-modules. 

\subsection{Drinfeld modules and the moduli schemes}\label{sec: Alg-geo.1}

First of all, we recall briefly the properties of Drinfeld modules which we need, and refer the readers to \cite{Dri} and \cite{Lau} for further details.
Fix a place $\infty$ of the global function field $k$ (with odd characteristic $p$). Denote by $A$ the ring of functions in $k$ regular outside $\infty$. Let $F$ be an $A$-field, i.e.\ $F$ is a field together with a ring homomorphism $\iota:A \rightarrow F$. We can identify $\End(\GG_{a/F})$ with the twisted polynomial ring $F\{\tau\}$, where $\tau : \GG_{a/F} \rightarrow \GG_{a/F}$ is the Frobenius map $(x \mapsto x^p)$ and $\tau a = a^p \tau$ for every $a \in F$. We have two homomorphisms $\epsilon : F \rightarrow F\{\tau\}$, $\epsilon(a) := a$, and $D_{\tau} : F\{ \tau\} \rightarrow F$, $D_{\tau}(\sum_i a_i \tau^i):= a_0$.

\begin{defn}\label{defn: Alg-geo.1}
Suppose an $A$-field $F$ and a positive integer $r$ is given.\\
$(1)$ A \text{\it Drinfeld $A$-module over $F$ of rank $r$} is a ring homomorphism $\phi: A \rightarrow F\{\tau\}$ such that $\iota = D_{\tau} \circ \phi$, $\phi \neq \epsilon \circ \iota$, and $p^{\deg_{\tau} \phi_a} = |a|_{\infty}^r$ for every $a \in A$.\\
$(2)$ Let $S$ be a scheme over $A$. A \text{\it Drinfeld $A$-module over $S$ of rank $r$} is a line bundle $L$ over $S$ together with a homomorphism $\phi: A \rightarrow \End(L)$ (where $\End(L)$ is the ring of endomorphisms of $L$ as a group scheme over $S$) such that
\begin{itemize}
\item[(i)] for any $a \in A$ the differential of $\phi_a$ is multiplication by $a$;
\item[(ii)] given a field $F$ with a morphism $\Spec(F) \rightarrow S$, the corresponding homomorphism $A \rightarrow F\{ \tau\}$ is a Drinfeld $A$-module over $F$ of rank $r$.
\end{itemize}
$(3)$ A morphism $f: (L,\phi) \rightarrow (L',\phi')$ is a homomorphism from $L$ to $L'$ (as group schemes over $S$) satisfying that $ f \cdot \phi_a = \phi'_a \cdot f$ for every $a \in A$.
\end{defn}

Let $(L,\phi)$ be a Drinfeld $A$-module over $S$ of rank $r$. For a non-zero ideal $I \lhd A$, let 
$$\phi[I]:= \bigcap_{a \in I} \ker\big (\phi_a\big),$$
which is a finite flat group scheme over $S$. 
Let $\Char_A(S)$ be the image of $S$ in $\Spec(A)$ (called the \text{\it $A$-characteristic} of $S$).
Then $\phi[I]$ is \'etale over $S$ if $\Char_A(S)$ does not intersect $V(I):= \{ \text{prime } \pfk \lhd A \mid \pfk \supset I\}$. 

\begin{defn}\label{defn: Alg-geo.2}
Given a Drinfeld $A$-module $(L,\phi)$ over $S$ of rank $r$,
a \text{\it structure of level $I$ on $L$} is an $A$-module homomorphism $\ell: (I^{-1}/A)^r \rightarrow \Mor(S,L)$ such that for every $\pfk \in V(I)$, $\phi[\pfk]$ (as a divisor of $L$) coincides with the sum of the divisors $\ell(a)$, $a \in (\pfk^{-1}/A)^r$.
\end{defn}

\begin{rem}\label{rem: Alg-geo.3}
If $\Char_A(S)$ does not intersect $V(I)$, then a structure of level $I$ gives an isomorphism from the constant group scheme $(I^{-1}/A)^r$ over $S$ to $\phi[I]$.
\end{rem}

Fix a positive integer $r$ and a non-zero ideal $I$ of $A$. For each scheme $S$ over $A$, 
let $\Mcal_{A,I}^r(S)$ be the category whose objects are the triples $(L,\phi,\ell)$, where $(L,\phi)$ is a Drinfeld $A$-module over $S$ of rank $r$ and $\ell$ is a structure of level $I$ on $L$, and the morphisms are the isomorphism between triples. We then obtain a fibered category $\Mcal_{A,I}^r$ over the category $\Sch_{/ A}$ of schemes over $A$. 

\begin{thm}\label{thm: Alg-geo.4}
\text{\rm (cf.\ \cite{Dri} and \cite{Lau})} $(1)$ When $0 \neq I$ and $V(I)$ contains more than one element, $\Mcal_{A,I}^r$ is represented by a scheme of finite type over $A$.\\
$(2)$ $\Mcal_{A,A}^r$ is representable by a Deligne-Mumford algebraic stack of finite type over $A$.\\
$(3)$ Denote by $M_A^1$ the corresponding set-valued functor of isomorphism classes of objects of $\Mcal_{A,A}^1$. Then $M_A^1$ has a coarse moduli scheme $\mathbf{M}_A = \Spec(O_{H_A})$, where $H_A$ is the Hilbert class field of $A$ (i.e.\ $H_A$ is the maximal unramified abelian extension over $k$ in which $\infty$ splits completely) and $O_{H_A}$ is the integral closure of $A$ in $H_A$.
\end{thm}

\subsection{Special morphism}\label{sec: Alg-geo.2}

Let $K$ be an imaginary (with respect to $\infty$) quadratic field over $k$ and $O_K$ denotes the integral closure of $A$ in $K$.
Let $S$ be a scheme over $O_K$. Given a Drinfeld $O_K$-module $(L,\phi)$ over $S$ of rank one and a non-zero ideal $\Ifk$ of $O_K$, one can associate a rank one Drinfeld $O_K$-module structure $\phi^{\Ifk}$ on $L^{\Ifk}:= L/\phi[\Ifk]$ such that the canonical homomorphism $u_{\Ifk}: L\rightarrow L^{\Ifk}$ is a morphism of Drinfeld modules (i.e.\ $u_{\Ifk} \cdot \phi_a = \phi^{\Ifk}_a \cdot u_{\Ifk}$ for every $a \in A$). Note that $\phi$ and $\phi^{\Ifk}$ can be viewed as Drinfeld $A$-modules of rank $2$ over $S$ \lq\lq with complex multiplication by $O_K$.\rq\rq

Recall that $\Dcal_{\alpha} = K + K j_{\alpha}$ where $j_{\alpha}^2 = -\alpha$ and $j_{\alpha} a = \bar{a} j_{\alpha}$ for every $a \in K$. Fix an embedding $\Dcal_{\alpha} \hookrightarrow \Mat_2(K)$ defined by
\begin{equation}
\label{eqn: Alg-geo.t}
 a_1 + a_2 j_{\alpha} \longmapsto \begin{pmatrix} a_1 & a_2 \\ -\alpha \overline{a_2} & \overline{a_1}\end{pmatrix}.
 \end{equation}
Let $$O_{\Dcal_{\alpha}}:= \Dcal_{\alpha} \cap \Mat_2(O_K).$$
We extend $\phi$ (resp.\ $\phi^{\Ifk}$) to a homomorphism from $\Mat_2(O_K)$ into $\End(L^{\oplus 2})$ (resp.\ $\End(L^{\Ifk, \oplus 2}))$ in the canonical way, and set
$$ \Hom_{\Dcal_{\alpha}}\big((L^{\Ifk, \oplus 2}, \phi^{\Ifk}),(L^{\oplus 2},\phi)\big):= \{ f \in \Hom(L^{\Ifk,\oplus 2},L^{\oplus 2}) : f \phi^{\Ifk}_d = \phi_d f,\ \forall d \in O_{\Dcal_{\alpha}}\}.$$
Every $f \in \Hom_{\Dcal_{\alpha}}\big((L^{\Ifk, \oplus 2}, \phi^{\Ifk}),(L^{\oplus 2},\phi)\big)$ is called a $\Dcal_{\alpha}$-morphism from $(L^{\Ifk,\oplus 2},\phi^{\Ifk})$ to $(L^{\oplus 2},\phi)$.

\begin{defn}\label{defn: Alg-geo.5}
A $\Dcal_{\alpha}$-morphism $f : (L^{\Ifk,\oplus 2},\phi^{\Ifk}) \rightarrow (L^{\oplus 2},\phi)$  is called \text{\it special} if $f \cdot \phi^{\Ifk}_a = \phi_{\bar{a}} \cdot f$ for every $a \in O_K$, where $a \mapsto \bar{a}$ is the non-trivial automorphism of $K$ over $k$, and $O_K$ embeds into $\Mat_2(O_K)$ by $a \mapsto \begin{pmatrix}a&0\\0&a\end{pmatrix}$.
\end{defn}

\begin{rem}\label{rem: Alg-geo.6}
Given a Drinfeld $O_K$-module $(L,\phi)$ over $S$ of rank one and a non-zero ideal $\Ifk$ of $O_K$, let $\SP_S^{(\alpha)} (L,\phi,\Ifk)$ be the set of special $\Dcal_\alpha$-morphisms from $(L^{\Ifk,\oplus 2},\phi^{\Ifk})$ to $(L^{\oplus 2},\phi)$.
Let $u_{\Ifk}^{\oplus 2} : L^{\oplus 2} \rightarrow L^{\Ifk,\oplus 2}$ be the morphism induced by $u_{\Ifk}$.
Then $\SP_S^{(\alpha)}(L,\phi,\Ifk)\cdot u_{\Ifk}^{\oplus 2}$ is contained in the ring $\End_{\Dcal_{\alpha}}(L^{\oplus 2},\phi)$ of $\Dcal_{\alpha}$-endomorphisms of $(L^{\oplus 2},\phi)$ and
$$(f \cdot u_{\Ifk}^{\oplus 2}) \cdot \phi_a = \phi_{\bar{a}} \cdot (f \cdot u_{\Ifk}^{\oplus 2})$$ 
for every $f \in \SP_S^{(\alpha)}(L,\phi,\Ifk)$ and $a \in O_K$.
Take $S = \Spec(\kappa)$ where $\kappa$ is an algebraically closed $O_K$-field. 
\begin{itemize}
\item[(1)] When $\Char_{A}(\kappa) = (0)$, it is known that $\End_A(L,\phi) = O_K$.
Therefore we get $\End_{\Dcal_{\alpha}}(L^{\oplus 2}, \phi) = O_K$ and $\SP_{\kappa}^{(\alpha)}(L,\phi,\Ifk) = 0$.
\item[(2)] Suppose $\Char_A(\kappa) = \pfk$ splits in $K$.
Then $(L,\phi)$ is not supersingular (as a Drinfeld $A$-module over $\kappa$). Therefore $\End_A(L,\phi)\otimes_A k = K$ (cf.\ \cite{Gek}, also see \cite[Proposition 1.7]{Dri}) and $\SP_{\kappa}^{(\alpha)}(L,\phi,\Ifk) = 0$.
\item[(3)] When $\Char_A(\kappa) = \pfk$ does not split in $K$, $(L,\phi)$ is supersingular and $\End_A(L,\phi) = R$ is a maximal $A$-order of the quarternion algebra $\Dcal_{\pfk}$ over $k$ ramified precisely at $\pfk$ and $\infty$. 
Write $\Dcal_{\pfk} = K + Kj$ where $j^2 \in k^{\times}$ and $ja = \bar{a} j$ for every $a$ in $F$.
We embed $\Mat_2(K)$ into $\Mat_2(\Dcal_{\pfk})$ via $K \hookrightarrow K+Kj = \Dcal_{\pfk}$. Together with (\ref{eqn: Alg-geo.t}), we obtain an embedding from $\Dcal_{\alpha}$ into $\Mat_2(\Dcal_{\pfk})$.
Let 
$$\widetilde{\Dcal}:= \{ M \in \Mat_2(\Dcal_{\pfk}): M d = d M,\ \forall d \in \Dcal_{\alpha}\} \quad \text{ and } \quad \widetilde{R}:= \Mat_2(R) \cap \widetilde{\Dcal}.$$
Writing $\widetilde{\Dcal} = K + K\tilde{j}$ where $\tilde{j}^2 \in k^{\times}$ and $\tilde{j}a = \bar{a} \tilde{j}$ for every $a \in F$, we then have
$$\SP_{\kappa}^{(\alpha)}(L,\phi,\Ifk) \cdot u_{\Ifk}^{\oplus 2} = \widetilde{R} \Ifk \cap K\tilde{j}.$$
Moreover, for every non-zero ideal $\Afk$ of $O_K$,
$$\End_{\Dcal}(L^{\Afk,\oplus 2},\phi^{\Afk}) = \Afk \widetilde{R} \Afk^{-1} (=: \widetilde{R}_{\Afk})$$ 
and
$$\SP_{\kappa}^{(\alpha)}(L^{\Afk,\oplus 2},\phi^{\Afk},\Ifk) \cdot u^{\Afk,\oplus 2}_{\Ifk} = \widetilde{R}_{\Afk} \Ifk \cap K\tilde{j},$$
where $u_{\Ifk}^{\Afk}$ is the canonical morphism from $(L^{\Afk},\phi^{\Afk})$ to $(L^{\Afk\Ifk},\phi^{\Afk\Ifk})$.
\end{itemize}
\end{rem}

\begin{defn}
Given a non-zero ideal $\Ifk \lhd O_K$ and $0\neq \beta \in A$, we define the fibered category $\Zcal(\Ifk,\beta)$ over $\Sch_{/A}$ as follows: for each scheme $S$ over $A$, $\Zcal(\Ifk,\beta)(S)$ is the category of triples $(L,\phi, b)$ where $(L,\phi)$ is a Drinfeld $O_K$-module over $S$ of rank one and $b \in \SP_S^{(\alpha)}(L,\phi,\Ifk)$ such that $$\tilde{b}^2 = \phi_{-\beta} \in \End_A(L,\phi), \quad \text{ where $\tilde{b}:= b \cdot u_{\Ifk}^{\oplus 2}$.}$$
The morphisms in the category $\Zcal(\Ifk,\beta)(S)$ are the isomorphisms between triples. We denote by $\pr: \Zcal(\Ifk,\beta) \rightarrow \Mcal_{O_K,O_K}^1$ the forgetful functor.
\end{defn}

\begin{prop}\label{prop: Alg-geo.7}
The forgetful functor $\pr: \Zcal(\Ifk,\beta) \rightarrow \Mcal_{O_K,O_K}^1$ is relatively representable, finite, and unramified.
\end{prop}

\begin{proof}
Given a scheme $U$ over $O_K$ and a Drinfeld $O_K$-module $(L,\phi)$ over $U$ of rank one,
we consider the set-valued functor $\Fcal_U$ on the category $\Sch_{/U}$ of schemes over $U$ defined by
\begin{equation}
\label{eqn: Alg-geo.1}
(f: S \rightarrow U) \longmapsto \big\{(L',\phi',b; \lambda) \mid (L',\phi',b) \in \Ob \big(\Zcal(\Ifk,\beta)(S)\big),\ \lambda : (f^*L,f^*\phi) \cong (L',\phi')\big\}/ \sim,
\end{equation}
where $f^*L := L \times_U S$, $f^*\phi : O_K \rightarrow \End(L)\rightarrow \End(f^*L)$, and $(L'_1,\phi'_1,b_1;\lambda_1) \sim (L'_2,\phi'_2,b_2;\lambda_2)$ if and only if there exists an isomorphism $\xi : (L'_1,\phi'_1,b_1)\cong (L'_2,\phi'_2,b_2)$ over $S$ such that $\xi \circ \lambda_1 = \lambda_2$.
The right-hand-side of (\ref{eqn: Alg-geo.1}) can be identified with
$$\big\{ b \in \SP_S^{(\alpha)}(f^*L,f^*\phi,\Ifk)\mid \tilde{b}^2 = -\beta\big\}.$$
Hence it can be shown that $\Fcal_U$ is representable (by a scheme $S_{\Fcal_U}$ over $U$), which says that $\pr$ is relatively representable.

To prove that $\pr$ is finite and unramified, it is equivalent to show $f_U: S_{\Fcal_U} \rightarrow U$ is finite and unramified for every scheme $U$ over $O_K$. First of all, we observe that $f_U$ is locally of finite presentation, quasi-finite, quasi-separated, and quasi-compact. By the rigidity theorem (cf.\ \cite[Proposition 4.1]{Dri}), the sheaf of relative differentials $\Omega_{S_{\Fcal_U}/U} = 0$.
Moreover, from the discussion in Remark~\ref{rem: Alg-geo.6}, we are able to show that $f_U$ satisfies the existence and the uniqueness of valuative criterion. Therefore $f_U$ is finite and unramified.
\end{proof}

The above proposition implies that $\Zcal(\Ifk,\beta)$ is representable by a Deligne-Mumford algebraic stack. Moreover, let $Z(\Ifk,\beta)$ be the corresponding set-valued functor of isomorphism classes of objects of $\Zcal(\Ifk,\beta)$. Then $Z(\Ifk,\beta)$ has a coarse moduli scheme $\mathbf{Z}(\Ifk,\beta)$. In fact, if we let $U = \mathbf{M}_{O_K} = \Spec(O_{H_{O_K}})$ and take a Drinfeld $O_K$-module $(L,\phi)$ over $O_{H_{O_K}}$ of rank one, then $\mathbf{Z}(\Ifk,\beta)$ is exactly the scheme $S_{\Fcal_K}$ in the proof of Proposition \ref{prop: Alg-geo.7}, and the induced morphism $\pr: \mathbf{Z}(\Ifk,\beta)\rightarrow \mathbf{M}_{O_K}$ is finite and unramified.

\begin{prop}\label{prop: Alg-geo.8}
Let $\kappa$ be an algebraically closed field together with a ring homomorphism $\iota: A \rightarrow \kappa$.\\
$(1)$ When $\ker (\iota) = (0)$, $\mathbf{Z}(\Ifk,\beta)(\kappa)$ is empty. In particular, $\mathbf{Z}(\Ifk,\beta)$ is an artinian scheme.\\
$(2)$ When $\ker (\iota) = \pfk$ where $\pfk$ is a prime ideal of $A$ split in $K$, $\mathbf{Z}(\Ifk,\beta)(\kappa)$ is empty.\\
$(3)$ When $\ker (\iota) = \pfk$ where $\pfk$ is a prime ideal of $A$ non-split in $K$,
$$\#\big(\mathbf{Z}(\Ifk,\beta)(\kappa)\big) = \sum_{[\Afk] \in \Pic(O_K)} \#\{ b \in \SP_{\kappa}^{(\alpha)}(L^{\Afk},\phi^{\Afk},\Ifk) \mid \tilde{b}^2 = \phi_{-\beta}^{\Afk}\},$$
where $(L,\phi)$ is a (any) chosen Drinfeld $O_K$-module over $\kappa$ of rank one, $\tilde{b}:= b \cdot u_{\Ifk}^{\Afk,\oplus 2}$ and $u_{\Ifk}^{\Afk}$ is the canonical morphism from $(L^{\Afk},\phi^{\Afk})$ to $(L^{\Afk \Ifk},\phi^{\Afk\Ifk})$.
\end{prop}

\begin{proof}
Take a Drinfeld $O_K$-module $(L,\phi)$ over $\kappa$ of rank one.
We can identify $\mathbf{M}_{O_K}(\kappa)$ with the set $\{(L^{\Afk},\phi^{\Afk})\mid \Afk \in \Pic(O_K)\}$, and the fiber of $\pr: \mathbf{Z}(\Ifk,\beta)(\kappa) \rightarrow \mathbf{M}_{O_K}(\kappa)$ of the point corresponding to $(L^{\Afk},\phi^{\Afk})$ can be identified with
$$\{ b \in \SP_{\kappa}(L^{\Afk},\phi^{\Afk},\Ifk) \mid \tilde{b}^2 = \phi_{-\beta}^{\Afk}\}.$$
Therefore the result follows from Remark \ref{rem: Alg-geo.6}.
\end{proof}

Let $\pfk$ be a prime ideal of $A$ which does not split in $K$. Let $\overline{\FF}_{\pfk}$ be an algebraic closure of $\FF_{\pfk}$. For $\xi \in \mathbf{M}_{O_K}(\overline{\FF}_{\pfk})$, it is known that (cf.\ \cite[Proposition 4.2]{Dri}) the local ring $\hat{\Ocal}_{\mathbf{M}_{O_K},\xi}$ (with respect to the \'etale topology)  is isomorphic to the ring $W(\overline{\FF}_{\pfk})$ of Witt vectors. 

\begin{prop}\label{prop: Alg-geo.9}
Let $\Ifk$ be a non-zero ideal of $O_K$ and $0 \neq \beta \in A$.
Let $\pfk$ be a prime ideal of $A$ which is not split in $K$.
For each point $\xi \in \mathbf{M}_{O_K}(\overline{\FF}_{\pfk})$ and  $\tilde{\xi} \in \mathbf{Z}(\Ifk,\beta) (=: \mathbf{Z})$ with $\pr(\tilde{\xi}) = \xi$, we have
$$\hat{\Ocal}_{\mathbf{Z},\tilde{\xi}} = W(\overline{\FF}_{\pfk})/(\varpi_{\pfk}^{\nu}),$$
where $\varpi_{\pfk}$ is a uniformizer in $W(\overline{\FF}_{\pfk})$,
$\nu = \nu_\pfk := \big(\ord_{\pfk}(\beta/\alpha N_{K/k}(\Ifk)) + 1\big)/f_{\pfk}$,
and $f_{\pfk}$ is the residue degree $[\FF_{\Pfk}:\FF_{\pfk}]$ of the unique prime $\Pfk$ of $O_K$ lying above $\pfk$.
\end{prop}

\begin{proof}
Proposition \ref{prop: Alg-geo.8} (1) says that $\hat{\Ocal}_{\mathbf{Z},\tilde{\xi}}$ is an artinian local ring.
Since $\pr: \mathbf{Z}(\Ifk,\beta) \rightarrow \mathbf{M}_{O_K}$ is finite and unramified, we must have
$$\hat{\Ocal}_{\mathbf{Z},\tilde{\xi}} = W(\overline{\FF}_{\pfk})/(\varpi_{\pfk}^{\nu}),$$
and $\nu$ can be determined by \cite[Proposition 4.3]{Gro}.
\end{proof}

Define the \text{\it degree of $\mathbf{Z}(\Ifk,\beta)$} as: 
$\deg\mathbf{Z}(\Ifk,\beta) := \sum\limits_{\tilde{\xi} \in \mathbf{Z}(\Ifk,\beta)} \ln \big(\#(\Ocal_{\mathbf{Z},\tilde{\xi}})\big)$.
Then Proposition \ref{prop: Alg-geo.7} and \ref{prop: Alg-geo.8} lead us to the following result.

\begin{thm}\label{thm: Alg-geo.10}
For each non-zero ideal $\Ifk$ of $O_K$ and $0 \neq \beta \in A$, $\deg\mathbf{Z}(\Ifk,\beta)$ is equal to
$$
\sum_{0 \neq \pfk \in \Spec(A)} \left(\ln q_{\pfk} \cdot \big(\ord_{\pfk}(\beta/\alpha N_{K/k}(\Ifk)) + 1\big) \cdot \sum_{[\Afk] \in \Pic(O_K)} \#\{b \in \SP_{\overline{\FF}_{\pfk}}(L^{\Afk},\phi^{\Afk},\Ifk) \mid \tilde{b}^2 = \phi_{-\beta}^{\Afk}\}\right).$$
\end{thm}

\begin{rem}\label{rem: Alg-geo.11}
$\deg\mathbf{Z}(\Ifk,\beta) = 0$ unless there exists a finite place $v_0$ of $k$ non-split in $K$ such that $\chi_{K,v_0}(\beta/\alpha) = -1$ and $\chi_{K,v}(\beta/\alpha) = 1$ for every $v \neq v_0, \infty$. In this case,
$$\deg\mathbf{Z}(\Ifk,\beta) = 
\ln q_{v_0} \cdot \big(\ord_{v_0}(\beta/\alpha N_{K/k}(\Ifk)) + 1\big) \cdot \sum_{[\Afk] \in \Pic(O_K)} \#\{b \in \SP_{\overline{\FF}_{v_0}}^{(\alpha)}(L^{\Afk},\phi^{\Afk},\Ifk) \mid \tilde{b}^2 = \phi_{-\beta}^{\Afk}\}.$$
\end{rem}

\subsection{Geometric interpretation of the Fourier coefficients of $\eta^{(\alpha)}$}\label{sec: Alg-geo.3}

Recall that in the beginning of this section we assumed that the additive character $\psi: \AA_k \rightarrow \CC^{\times}$ is chosen so that $\ord_v(\alpha) + \delta_v$ is even for every place $v$ inert in $K$. By Proposition \ref{prop: Fourier.6}, we know that for $y \in \AA_k^{\times}$ and $0 \neq \beta \in k$ with $\ord_v(y_v^2\beta)+\delta_v \geq 0$,
the Fourier coefficient $\eta^{(\alpha),*}_{\beta}(y)$ vanishes unless there exists a place $v_0$ of $k$ such that $\Diff(\beta,\Ccal_K^{(\alpha)}) = \{v_0\}$.

Suppose $v_0 \neq \infty$.
We then have $\chi_{K,v_0}(-\gamma) = -1$ and $\chi_{K,v}(-\gamma) = 1$ for every place $v \neq v_0, \infty$.
Since $\eta^{(\alpha),*}_{\beta}(y) = \eta^{(\alpha),*}_{a^2 \beta}(a^{-1}y)$ for every $a \in k^{\times}$, it suffices to consider that $\Ifk_y := \yfk^{-1} \overline{\Dfk}_{(\alpha)}^{-1}$ (where $\yfk$ and $\Dfk_{\alpha}$ are introduced in Lemma \ref{lem: Examp.8} and Remark \ref{rem: Examp.10} respectively) is an integral ideal of $O_K$ and $\beta \in A$.
Then Proposition \ref{prop: Examp.9} and Remark \ref{rem: Examp.10} imply that
\begin{eqnarray}
\eta^{(\alpha),*}_{\beta}(y)&= &-\frac{\chi_F(y)|y|_{\AA_k}\ln q_{v_0} }{f_{\infty} \cdot \#\Pic(A)}  \cdot  \big(\ord_{v_0}(y_{v_0}^2\beta) + \delta_{v_0}+1\big) \nonumber \\
&& \quad \quad  \quad \quad  \cdot \sum_{[\Afk] \in \Pic(O_F)} \{ b \in \widetilde{R}_{\Afk} \Ifk_y \cap Fj \mid b^2 = -\beta\}. \nonumber 
\end{eqnarray}
Therefore by Remark \ref{rem: Alg-geo.6} and Theorem \ref{thm: Alg-geo.10} we have:

\begin{cor}\label{cor: Alg-geo.12}
Take $0 \neq \beta \in A$ with $\Diff(\beta,\Ccal_K^{(\alpha)}) = \{v_0\}$, $v_0 \neq \infty$, and $y \in \AA_k$ such that $\ord_{v}(y_{v}^2\beta)+\delta_{v} \geq 0$ for every place $v$ of $k$. Suppose further that $\Ifk_y = \yfk^{-1} \overline{\Dfk}_{(\alpha)}^{-1}$ is an integral ideal of $O_K$.
Then 
$$\eta_{\beta}^{(\alpha),*}(y) = -\frac{\chi_K(y)|y|_{\AA_k}}{f_{\infty} \cdot \#\Pic(A)}\cdot \deg\mathbf{Z}(\Ifk_y,\beta),$$
where $f_{\infty}$ is the residue degree of $\infty$ in $K/k$.
\end{cor}

Let $\Xcal_{O_K}$ be the \lq\lq compactification\rq\rq\ of $\mathbf{M}_{O_K}$, that is, $\Xcal_{O_K}$ is the projective smooth model of the function field $H_{O_K}$. In particular, $\mathbf{M}_{O_K} = \Spec(O_{H_{O_K}})$ is an affine piece of $\Xcal_{O_K}$, and 
$$\Xcal_{O_K} \backslash \mathbf{M}_{O_K}  = \{\infty'_{\Afk} : [\Afk] \in \Pic(O_K)\},$$
where $\infty'_{\Afk}$, $[\Afk] \in \Pic(O_K)$, are the closed points lying above $\infty$.\\

Now, given a pair $(y,\beta)$ where $y \in \AA_k^{\times}$ and $\beta \in k^{\times}$, we define the special cycle
$\mathbf{z}(y,\beta)$ on $\Xcal_{O_K}$ as follows:
\begin{itemize}
\item When $(y,\beta)$ satisfies the assumption in Corollary~\ref{cor: Alg-geo.12},
we let $\mathbf{z}(y,\beta):= \mathbf{Z}(\Ifk_y,\beta)$.
\item When $v_0 = \infty$, we have $\beta/\alpha \in N_{K/k}(K^\times)$. Set 
$$\lambda_{\infty}:= \frac{\big(3+(-1)^{f_{\infty}} +q_{\infty}(1-(-1)^{f_{\infty}})\big)}{2(1+q_{\infty})}$$
and
\begin{eqnarray}
\mathbf{z}(y,\beta) & := & f_{\infty}^{-1}\cdot \big(\ord_{\infty}(y_{\infty}^2 \beta)+ \delta_{\infty}+\lambda_{\infty}\big) \nonumber \\
 & & \!\!\!\!\!\!\! \cdot \sum_{[\Afk] \in \Pic(O_K)} \#\{x \in \Afk \bar{\Afk}^{-1} \overline{\Ifk}_y : N_{K/k}(x) = \beta/\alpha \} \cdot \infty'_{\Afk}, \nonumber
 \end{eqnarray}
which is a divisor (with rational coefficients) on $\Xcal_{O_K}$.
\item For a general pair $(y,\beta)$ where $y \in \AA_k^{\times}$ and $\beta \in k^{\times}$,
put $\mathbf{z}(y,\beta):= 0$ if there exists a place $v$ of $k$ such that $\ord_v(y_v^2 \beta) + \delta_v < 0$. Suppose $\ord_v(y_v^2 \beta) + \delta_v \geq 0$ for every place $v$, choose $a \in k^{\times}$ such that $\beta':=a^2 \beta \in A$ and $\Ifk_{y'}$, where $y':= y a^{-1}$, is an integral ideal of $O_K$. Thus $(y',\beta')$ satisfies the assumption in Corollary~\ref{cor: Alg-geo.12}. We then put $\mathbf{z}(y,\beta) := \mathbf{z}(y',\beta')$.
\end{itemize}

Finally, by Remark~\ref{rem: Examp.10} (2) and Corollary~\ref{cor: Alg-geo.12} we arrive at:

\begin{thm}\label{thm: Alg-geo.13}
For every $y \in \AA_k^{\times}$ and $\beta \in k^{\times}$, we have
$$\eta^{(\alpha),*}_\beta(y) = -\frac{\chi_F(y)|y|_{\AA_k}}{f_{\infty} \cdot \#\Pic(A)}\cdot  \deg \mathbf{z}(y,\beta).$$
\end{thm}

\subsubsection{Geometric interpretation of $\eta_0^{(\alpha),*}(y)$}
Under the assumption that $\ord_v(\alpha) + \delta_v$ is even for every place $v$ of $k$, 
the formula of $\eta_0^{(\alpha),*}(y)$ in Lemma~\ref{lem: Examp.2} becomes:
\begin{eqnarray}
\eta^{(\alpha),*}_{0}(y) &=& 2\chi_K(y)|y|_{\AA_k}L(0,\chi_K) \nonumber \\
&& \!\!\!\!\!\!\!\!\!\!\!\!\!\!\! \cdot \left[ \ln|y|_{\AA_k}-\Big([\FF_K:\FF_q](g_K-1)-(g_k-1)\Big)\ln q - \frac{L'(0,\chi_K)}{L(0,\chi_K)}
+  \frac{q_\infty^{(f_\infty-1)}-1}{q_\infty+1}\ln q_\infty\right]. \nonumber
\end{eqnarray}
On the other hand, let $\phi$ be a Drinfeld $A$-modules over $\bar{k}$ of rank $2$ with complex multiplication by $O_K$. The logarithmic derivative of $L(s,\chi_K)$ at $s=0$ is connected with the \lq\lq Taguchi height\rq\rq\ $\tilde{h}_{\text{Tag}}(\phi)$ of $\phi$ (cf.\ \cite[Corollary 0.2]{Wei3}):
$$
\frac{L'(0,\chi_K)}{L(0,\chi_K)} = \ln q \cdot \Big[ 2(g_k-1) - [\FF_K:\FF_k](g_K-1) + \frac{\deg \infty}{f_\infty} \Big] \ln q - \frac{\zeta_A'(0)}{\zeta_A(0)} - 2\tilde{h}_{\text{\rm Tag}}(\phi).
$$
Here $\zeta_A(s)$ is the zeta function of $A$:
$$\zeta_A(s) = \prod_{v \neq \infty} \frac{1}{1- q_v^{-s}}, \quad  \re(s)>1.$$
Recall the formula of $L(0,\chi_K)$ in Remark~\ref{rem: Examp.3} (2):
$$L(0,\chi_K) = \frac{\#\Pic(O_K)}{f_\infty \#\Pic(A)}.$$
Therefore we obtain that

\begin{lem}\label{lem: Alg-geo.14}
\begin{eqnarray}
\eta^{(\alpha),*}_{0}(y) &=& \frac{2\chi_K(y)|y|_{\AA_k}\#\Pic(O_K)}{f_\infty \#\Pic(A)} \nonumber \\
&& \!\!\!\!\!\!\!\!\!\!\!\!\!\!\! \cdot \left[ \ln|y|_{\AA_k}  - 2\tilde{h}_{\text{\rm Tag}}(\phi)  -(g_k-1)\ln q - \frac{\zeta_A'(0)}{\zeta_A(0)} 
+  \frac{(-1)^{f_\infty} q_\infty +1 - 2^{f_\infty}}{f_\infty(q_\infty+1)}\ln q_\infty\right]. \nonumber
\end{eqnarray}
\end{lem}


\subsection*{Acknowledgements} 
The author is grateful to Jing Yu for his steady interest and encouragements.
Part of this work was carried out while the author was visiting Institut des Hautes \'Etudes Scientifiques in Bures-sur-Yvette, France. He wishes to thank the institute for kind hospitality and nice working conditions.

\end{document}